\documentclass[10pt,journal,compsoc]{IEEEtran}
\usepackage{graphicx}
\usepackage{url}
\usepackage{multicol}
\usepackage[cmex10]{amsmath}
\usepackage{amssymb}
\usepackage{macro}
\usepackage{myalgo}
\usepackage{hyperref}
\usepackage{amsthm}

\usepackage{color}

\newtheorem{theorem}{Theorem}

\newtheorem{property}{Property}
\newtheorem{definition}{Definition}
\newtheorem{lemma}{Lemma}
\newtheorem{proposition}{Proposition}
\newtheorem{remark}{Remark}
\newtheorem{example}{Example}

\begin{document}
%
\title{Certified Roundoff Error Bounds using Bernstein Expansions and Sparse Krivine-Stengle Representations}
%

\author{Victor Magron,~
        Alexandre Rocca,~
        and~Thao Dang
\IEEEcompsocitemizethanks{
\IEEEcompsocthanksitem All authors are affiliated to UGA, VERIMAG/CNRS, 700 avenue centrale 38400 Saint Martin D'H{\`e}res, France.\protect\\
contact E-mail: forename.surname@univ-grenoble-alpes.fr
\IEEEcompsocthanksitem A.~Rocca is also affiliated to UGA-Grenoble 1/CNRS, TIMC-IMAG, UMR 5525, Grenoble, F-38041, France.
\IEEEcompsocthanksitem The first author has been partially supported by the LabEx PERSYVAL-Lab (ANR-11-LABX-0025-01) funded by the French program ``Investissement d'avenir'' and by the European Research Council (ERC) ``STATOR'' Grant Agreement nr. 306595. The third author has been partially supported by the ANR MALTHY project (grant ANR-12-INSE-003).
}
}

\IEEEtitleabstractindextext{
\begin{abstract}
Floating point error is a drawback of embedded systems implementation that is difficult to avoid. Computing rigorous upper bounds of roundoff errors is absolutely necessary for the validation of critical software. This problem of computing rigorous upper bounds is even more challenging when addressing non-linear programs. In this paper, we propose and compare two new algorithms based on Bernstein expansions and sparse Krivine-Stengle representations, adapted from the field of the global optimization, to compute upper bounds of roundoff errors for programs implementing polynomial \new{and rational} functions. We also provide the convergence rate of these two algorithms.
We release two related software package $\fpbern$ and $\fpkristen$, and compare them with the state-of-the-art tools.
We show that these two methods achieve competitive performance, while providing accurate upper bounds by comparison with the other tools.
\end{abstract}
}
\maketitle

\IEEEdisplaynontitleabstractindextext

%
%
\IEEEpeerreviewmaketitle
\IEEEraisesectionheading{\section{Introduction}\label{introduction_intro}}

\IEEEPARstart{T}heoretical models, algorithms, and programs are often analyzed and designed in real algebra. However, their implementation on computers often uses floating point algebra: this conversion from real numbers and their operations to floating point is not without errors. Indeed, due to finite memory and binary encoding in computers, real numbers cannot be exactly represented by floating point numbers. Moreover, numerous properties of the real algebra are not preserved such as associativity.

The consequences of such imprecisions become particularly significant in safety-critical systems, especially in embedded systems which often include control components implemented as computer programs. When implementing an algorithm designed in real algebra, and initially tested on computers with single or double floating point precision, one would like to ensure that the roundoff error is not too large on more limited platforms (small processor, low memory capacity) by computing their accurate upper bounds.

For programs implementing linear functions, SAT/SMT solvers as well as affine arithmetic are efficient tools to obtain good upper bounds. When extending to programs with non-linear polynomial \new{or rational} functions, the problem of determining a precise upper bound becomes substantially more difficult, since polynomial optimization problems are in general NP-hard~\cite{laurent2009sums}. We can cite at least three closely related and recent frameworks designed to provide upper bounds of roundoff errors for non-linear programs. \texttt{FPTaylor} \cite{fptaylor} is a tool based on Taylor-interval methods, while \texttt{Rosa} \cite{rosa} combines SMT with interval arithmetic.  \texttt{Real2Float} \cite{real2float} relies on Putinar representations of positive polynomials while exploiting sparsity in a similar way as the second method that we propose in this paper. 

The contributions of this paper are two methods, coming from the field of polynomial optimization, to compute upper bounds on roundoff errors of \new{programs involving polynomial or rational functions}.  The first method is based on Bernstein expansions of polynomials, while the second relies on sparse Krivine-Stengle certificates for positive polynomials. 
In practice, these methods (presented in Section~\ref{contributions}) provide accurate bounds at a reasonable computational cost. 
Indeed, the size of the Bernstein expansions used in the first method as well as the size of the LP relaxation problems considered in the second method are both linear w.r.t.~the number of roundoff error variables. 
\subsection{Overview}
\label{overview}
Before explaining in detail each method, let us first illustrate the addressed problem on an example. Let $f$ be the degree two polynomial defined by:
\[ 
f(x) := x^2 - x~,\quad \forall x\in {X} = [0,1].
\]
When approximating the value of $f$ at a given real number $x$, one actually computes the floating point result $\hat{f} = \hat{x}\otimes\hat{x} \ominus \hat{x}$, with all the real operators $+$,$-$,$\times$ being substituted by their associated floating point operators $\oplus$, $\ominus$, $\otimes$, and $x$ being represented by the floating point number $\hat{x}$ (see Section \ref{preliminaries_floating_point} for more details on floating point arithmetics). A simple rounding model consists of introducing an error term $e_i$ for each floating point operation, as well as for each floating point variable. For instance, $\hat{x}\otimes\hat{x}$ corresponds to $((1+e_1) \, x \,  (1+e_1) \, x) \, (1+e_2)$, where $e_1$ is the error term between $x$ and $\hat{x}$, and $e_2$ is the one associated to the operation $\otimes$.
Let $\eb$ be the vector of all error terms $e_i$. Given $e_i \in [ - \varepsilon ,\varepsilon ]$ for all $i$, with $ \varepsilon $ being the machine precision, we can write the floating point approximation $\hat{f}$ of $f$ as follows:
\[ 
\hat{f}(x,\eb) = (((1+e_1)x  (1+e_1) x)  (1+e_2) - x (1+e_1)) (1+e_3).
\]
Then, the absolute roundoff error is defined by:
\[
r(x,\eb) := \max\limits_{\begin{subarray}{1}x\in [0,1]\\ \eb\in[ - \varepsilon ,\varepsilon ]^3\end{subarray}}(|\hat{f}(x,\eb) - f(x)|)~~.
\]
However, we can make this computation easier with a slight approximation: $|\hat{f}(x,\eb)-f(x)| \leq |l(x,\eb)| + |h(x,\eb)|$ with $l(x,\eb)$ being the sum of the terms of $(\hat{f}(x,\eb)-f(x))$ which are linear in $\eb$, and $h(x,\eb)$ the sum of the terms which are non-linear in $\eb$. The term $|h(x,\eb)|$ can then be over-approximated by $O(|\eb|^2)$ which is \emph{in general} negligible compared to $|l(x,\eb)|$, and can be bounded using standard interval arithmetic. For this reason, we focus on computing an upper bound of $|l(x,\eb)|$. In the context of our example, $l(x,\eb)$ is given by:
\begin{align}
\label{eq:l}
l(x,\eb) = (2x^2-x) e_1 + x^2 e_2 + (x^2-x) e_3.
\end{align} 
We divide each error term $e_j$ by $\varepsilon$, and then consider the (scaled) linear part $l':=\frac{l}{\varepsilon}$ of the roundoff error with the error terms $\eb\in [-1,1]^3$.
For all $x\in [0,1]$, and $\eb\in [-1,1]^3$, one can easily compute a valid upper bound of $|l'(x,\eb)|$ with interval arithmetic. Using the same notation for elementary operations $+,-,\times$ in interval arithmetic, \new{one has $l'(x,\eb) \in ([-0.125,1]\times[-1,1]+[0,1]\times[-1,1]+[-0.25,0]\times[-1,1])=[-2.25,2.25]$,} yielding $|l(x,\eb)|\leq 2.25\varepsilon$.\\
Using the first method based on Bernstein expansions detailed in Section \ref{contributions_bernstein}, we obtained $2\varepsilon$ as an upper bound of $|l(x,\eb)|$ after $0.23$s of computation using \texttt{FPBern(b)} a rational arithmetic implementation. With the second method based on sparse Krivine-Stengle representation detailed in Section \ref{contributions_handelman}, we also obtained an upper bound of $2\varepsilon$ in $0.03$s.\\
 \new{Although on this particular example, the method based on sparse Krivine-Stengle representations appears to be more time-efficient,  in general the computational cost of the method based on Bernstein  expansions is lower. For this example, the bounds provided by both methods are tighter than the ones determined by interval arithmetic.}
We emphasize the fact that the bounds provided by our two methods can be certified. Indeed, in the first case, the Bernstein coefficients (see Sections \ref{preliminaries_bernstein} and \ref{contributions_bernstein}) can be computed either with rational arithmetic or certified interval arithmetic to ensure guaranteed values of upper bounds. In the second case, the nonnegativity certificates are directly provided by sparse Krivine-Stengle representations.

\subsection{Related Works}
\label{related_works}
We first mention two tools, based on positivity certificates, to compute roundoff error bounds. The first tool, related to~\cite{Constantinides}, relies on an approach  similar to our second method. It uses dense Krivine-Stengle representations of positive polynomials to cast the initial problem as a finite dimensional LP problem. To reduce the size of this possibly large LP, \cite{Constantinides} provides heuristics to eliminate some variables and constraints in the dense representation. However, this approach has the main drawback of loosing the property of convergence toward optimal solutions of the initial problem.
Our second method uses sparse representations and is based on the previous works~\cite{schweighofer06} and~\cite{sbsos}, allowing to ensure the convergence towards optimal solutions while greatly reducing the computational cost of LP problems. Another tool, \texttt{Real2Float} \cite{real2float}, exploits sparsity in the same way while using Putinar representations of positive polynomials, leading to solving semidefinite (SDP) problems. Bounds provided by such SDP relaxations are in general more precise than LP relaxations~\cite{Lasserre2009Moments}, but their solving cost is higher.\newline
Several other tools are available to compute floating point roundoff errors. 
SMT solvers are efficient when handling linear programs, but often provide coarse bounds for non-linear programs, e.g.~when the analysis is done in isolation~\cite{rosa}. The \texttt{Rosa}~\cite{rosa} tool is a solver mixing SMT and interval arithmetic which can compile functional \texttt{SCALA} programs implementing non-linear functions (involving $/,\surd$ operations and polynomials) as well as conditional statements. SMT solvers are theoretically able to output certificates which can be validated externally afterwards. 
\texttt{FPTaylor} tool \cite{fptaylor} relies on \emph{Symbolic Taylor expansion} method, which consists of a branch and bound algorithm based on interval arithmetic. 
Bernstein expansions have been extensively used to handle systems of polynomial equations and inequalities, as well as polynomial optimization (see for example \cite{mourrain2009,smithThesis,dreossiHSCC,munoz13}). 
\new{In~\cite{NGSM12}, the authors provide a method to extend the range of Bernstein expansions to handle the case of rational function over a box. This approach consists of expanding both numerators and denominators.}
Yet, to the best of our knowledge, there is no tool based on Bernstein expansions in the context of roundoff error computation.
The \texttt{Gappa} tool provides certified bounds with elaborated interval arithmetic procedure relying on multiple-precision dyadic fractions. 
The static analysis tool \texttt{FLUCTUAT} \cite{fluctuat}  performs forward computation (by contrast with optimization) to analyze floating point \texttt{C} programs. Both \texttt{FLUCTUAT} and \texttt{Gappa} use a different rounding model (see Section \ref{preliminaries_floating_point}), also available in \texttt{FPTaylor}, that we do not handle in our current implementation.
Some tools also allow formal validation of certified bounds. \texttt{FPTaylor}, \texttt{Real2Float} \cite{real2float}, as well as \texttt{Gappa} \cite{gappa} provide formal proof certificates, with \texttt{HOL-Light} \cite{hollight} for the first case, and \texttt{Coq} \cite{CoqProofAssistant} for the two other ones. 

\subsection{Key Contributions}
\label{key_contributions}
Here is a summary of our key contributions:
\begin{itemize}
\if{
OLD
\item[$\blacktriangleright$] We present two new methods to compute upper bounds of floating point roundoff errors for programs implementing multivariate polynomial functions with input variables constrained to boxes. The first one is based on Bernstein expansions and the second one relies on sparse Krivine-Stengle representations.
We also propose a theoretical framework to guarantee the validity of upper bounds computed with both methods (see Section~\ref{contributions}). In addition, we give an alternative shorter proof in Section \ref{preliminaries_handelman} for the existence of Krivine-Stengle representations for sparse positive polynomials \new{(proof of Theorem~\ref{spkrivrep_th})}. 
}\fi
\item[$\blacktriangleright$] \new{We present two new algorithms to compute upper bounds of floating point roundoff errors for programs involving multivariate polynomials. The first algorithm is based on Bernstein expansions and handle programs implementing rational functions with box constrained input sets. The second algorithm relies on sparse Krivine-Stengle representations and handles programs implementing polynomial functions with input sets defined as conjunctions of finitely many polynomial inequalities.}
We also propose a theoretical framework to guarantee the validity of upper bounds computed with both algorithms (see Section~\ref{contributions}). In addition, we give an alternative shorter proof in Section \ref{preliminaries_handelman} for the existence of Krivine-Stengle representations for sparse positive polynomials \new{(proof of Theorem~\ref{spkrivrep_th}). We study in Section~\ref{complexity} the convergence rate of the two algorithms towards the maximal value of the linear part of the roundoff error.}

%

\item[$\blacktriangleright$] We release two software packages based on each algorithm. The first one, called \texttt{FPBern}\footnote{ \url{https://github.com/roccaa/FPBern} }, computes the bounds using Bernstein expansions, with two modules built on top of the \new{\texttt{C++} }software related to \cite{dreossiHSCC}: \texttt{FPBern(a)} is a module using double precision floating point arithmetic while \new{the second module \texttt{FPBern(b)} uses }rational arithmetic. 
The second one \texttt{FPKriSten}\footnote{ \url{https://github.com/roccaa/FPKriSten} } computes the bounds using Krivine-Stengle representations in \texttt{Matlab}. \texttt{FPKriSten} is built on top of the implementation related to \cite{sbsos}.
\item[$\blacktriangleright$] We compare our two methods implemented in \texttt{FPBern} and \texttt{FPKriSten} to three state-of-the-art methods. Our new methods have precisions comparable to that of these tools (\texttt{Real2Float, Rosa, FPTaylor}). At the same time, \texttt{FPBern(a)} \new{and \texttt{FPBern(b)} show }an important time performance improvement, while \texttt{FPKriSten} has similar time performances compared with the other tools, yielding promising results. 
\end{itemize}
\new{This work is the follow-up of our previous contribution~\cite{arith24}. The main novelties, both theoretical and practical, are the following: in~\cite{arith24}, we could only handle polynomial programs with box input constraints. For $\fpbern$, the extension to rational functions relies on~\cite{NGSM12}. We brought major updates to the \texttt{C++} code of $\fpbern$~(b). For $\fpkristen$, an extension to semialgebraic input sets was already theoretically possible in~\cite{arith24} with the hierarchy of LP relaxations based on sparse Krivine-Stengle representations, and in this current version, we have updated Section~\ref{contributions_handelman} accordingly to handle this more general case. We have carefully implemented this extension in our software package $\fpkristen$ in order to not compromise efficiency. 
Additional $15$ benchmarks provided in Section~\ref{implementation} illustrate the abilities of both algorithms to tackle a wider range of numerical programs. Another novelty is the complexity analysis of the two algorithms in the case of polynomial programs with box constrained input sets. This study is inspired by the framework presented in~\cite{deKlerk10Error}, yielding error bounds in the context of polynomial optimization over the hypercube.
}

The rest of the paper is organized as follows: in Section \ref{preliminaries}, we give basic background on floating point arithmetic, Bernstein expansions and Krivine-Stengle representations. In Section \ref{contributions} we describe the main contributions, that is the computation of roundoff error bounds using Bernstein expansions and sparse Krivine-Stengle representations. Finally, in Section \ref{implementation} we compare the performance and precision of our two methods with the existing tools, and show the advantages of our tools.
\section{Preliminaries}
\label{preliminaries}
We first recall useful notation on multivariate calculus. 
For $\mathbf{x} = (x_1,\ldots, x_n)\in \mathbb{R}^n$ and the multi-index $\mathbf{\ab} = (\alpha_1,\ldots, \alpha_n)\in \mathbb{N}^n$, we denote by
 $\xb^{\ab}$ the product $\prod_{i = 1}^n x_i^{\alpha_i}$. We also define $|\ab| = |\alpha_1| + \ldots + |\alpha_n|$, $\bzero=(0,\ldots,0)$ and $\bone = (1,\ldots,1)$.\newline
The notation $\sum_{\ab}$ is the nested sum $\sum_{\alpha_1}\ldots\sum_{\alpha_n}$. Equivalently $\prod_{\ab}$ is equal to the nested product $\prod_{\alpha_1}\ldots\prod_{\alpha_n}$. \newline
Given another multi-index $\mathbf{d} = (d_1,\ldots, d_n)\in \mathbb{N}^{n}$, the inequality $\ab < \mathbf{d}$ (resp.~$\ab \leq \mathbf{d}$) means that the inequality holds for each sub-index: $\alpha_1 < d_1, \ldots, \alpha_n < d_n$ (resp.~$\alpha_1 \leq d_1, \ldots, \alpha_n \leq d_n$). Moreover, the binomial coefficient $\binom{\mathbf{d}}{\ab}$ is the product $\prod_{i = 1}^n \binom{d_i}{\alpha_i}$.\newline
Let $\mathbb{R}[\xb]$ be the vector space of multivariate polynomials.  Given $f\in\mathbb{R}[\xb]$,  we associate a {\em multi-degree} $\mathbf{d} = (d_1, \dots, d_n)$ to $f$, with each $d_i$ standing for the degree of $f$ with respect to the variable $x_i$. 
Then, we can write $f(\mathbf{x}) = \sum_{\gb\leq\mathbf{d}}a_{\gb} \xb^{\gb}$, with $a_{\gb}$ (also denoted by $(f)_{\gb}$) being the coefficients of $f$ in the monomial basis and each $\gb\in \mathbb{N}^n$ is a multi-index. The degree $d$ of $f$ is given by $d := \max_{\{\gb : a_{\gb}\neq 0\}} | \gb |$. As an example, if $f(x_1,x_2) = x_1^4x_2 +x_1^1x_2^3$ then $\mathbf{d} = (4,3)$ and $d = 5$. For the polynomial $l$ used in Section~\ref{overview}, one has $\mathbf{d} = (2,1,1,1)$ and $d = 3$.
%
\subsection{Floating Point Arithmetic}
\label{preliminaries_floating_point}
This section gives background on floating point arithmetic, inspired from material available in~\cite[Section 3]{fptaylor}. The IEEE754 standard \cite{IEEE754} defines a binary floating point number as a triple of significant, sign, and exponent (denoted by $sig,sgn,exp$) which represents the numerical value $(-1)^{sgn}\times sig \times 2^{exp}$. The standard describes 3 formats (32, 64, and 128 bits) which differ by the size of the significant and the exponent, as well as special values (such as NaN, the infinities).
Let $\mathbb{F}$ be the set of floating point numbers, the rounding operator is defined by the function $\textnormal{rnd}:\mathbb{R}\rightarrow\mathbb{F}$ which takes a real number and returns the closest floating point number rounded to the nearest, toward zero, or toward $\pm\infty$. A simple model of rounding is given by the following formula:
\[ \textnormal{rnd}(x) = x(1+e) + u,\]
with $|e|\leq\varepsilon$, $|u|\leq \mu$ and $eu=0$. The value $\varepsilon$ is the maximal relative error (given by the machine precision \cite{IEEE754}), and $\mu$ is the maximal absolute error for numbers very close to $0$. For example, in the single (32 bits) format, $\varepsilon$ is equal to $2^{-24}$ while $\mu$ equals $2^{-150}$. \new{\emph{In general} }$\mu$ is negligible compared to $\varepsilon$, thus we neglect terms depending on $u$ in the remainder of this paper.\newline
Given an operation op : $\mathbb{R}^n\rightarrow\mathbb{R}$, let $\textnormal{op}_{\textnormal{FP}}$ be the corresponding floating point operation. An operation is exactly rounded when
$\textnormal{op}_{\textnormal{FP}}(\xb) = \textnormal{rdn}(\textnormal{op}(\xb))$ for all $\xb \in \mathbb{R}^n$. \newline
In the IEEE754 standard the following operations are defined as exactly rounded: $+,-,\times,/,\surd,$ and the \texttt{fma} operation\footnote{The \texttt{fma} operator is defined by \texttt{fma}($x,y,z$)=$x \times y+z$.}. It follows that for these operations we have the continuation of the simple rounding model  $\textnormal{op}_{\textnormal{FP}}(\xb) =\textnormal{op}(\xb)(1+e)$. \newline
The previous rounding model is called ``simple" in contrast with more improved rounding model. \new{Given the function $\textnormal{pc}(x) = \max_{k\in\mathbb{Z}}\{2^k : 2^k < x\}$, then the improved rounding model is defined by:
$
\textnormal{op}_{\textnormal{FP}}(\xb) =\textnormal{op}(\xb) +  \textnormal{pc}(\textnormal{op}(\xb))
$, for all $\xb \in \mathbb{R}^n$.
As the function pc is piecewise constant, this rounding model needs design of algorithms based on successive subdivisions, which is not currently handled in our methods. In the tools \texttt{FLUCTUAT}\cite{fluctuat},} \texttt{Gappa}\cite{gappa}, and \texttt{FPTaylor} \cite{fptaylor}, combining branch and bound algorithms with interval arithmetic is adapted to roundoff error computation with such rounding model. 
\subsection{Bernstein Expansions of Polynomials}
\label{preliminaries_bernstein}
In this section we give background on the Bernstein expansion, which is important to understand the contribution detailed subsequently in Section \ref{contributions_bernstein}.
Given a multivariate polynomial $f \in \mathbb{R}[\xb]$, we recall how to compute a lower bound of $\underline{f} :=\min_{\xb\in \X}f(\xb)$ where $\X = [0, 1]^n$.
The next result can be retrieved in~\cite[Theorem 2]{berngarloff}:
\begin{theorem}[Multivariate Bernstein expansion]
Given a multivariate polynomial $f$ and $\kb\geq \db$ with $\db$ the multi-degree of $f$, then the Bernstein expansion of multi-degree $\kb$ of $f$ is given by:\\
\begin{equation}
f(\mathbf{x}) = \sum\limits_{\gb} f_{\gb} \xb^{\gb} = \sum\limits_{\ab\leq\kb} b_{\ab}^{(f)}\mathbf{B}_{\kb,\ab}(\mathbf{x}).
\end{equation}
where $b_{\ab}^{(f)}$ (also denoted by $b_{\ab}$ \new{when there is no confusion}) are the Bernstein coefficients (of multi-degree $\kb$) of $f$, and $\Bb_{\kb,\ab}(\xb)$ are the Bernstein basis polynomials defined by $\Bb_{\kb,\ab}(\xb) := \prod_{i = 1}^n B_{k_i,\alpha_i}(x_i)$ and $B_{k_i,\alpha_i}(x_i) := \binom{k_i} {\alpha_i} x_i^{\alpha_i}(1-x_i)^{k_i - \alpha_i}$. The Bernstein coefficients of $f$ are given as follows:
\begin{equation}
b_{\ab} = \sum_{\bb < \ab}\frac{ \binom{\ab} {\bb} }{ \binom{\kb} {\bb} } f_{\bb}, \quad \bzero\leq \ab \leq \kb.
\end{equation}
\end{theorem}
Bernstein expansions having numerous properties, we give only \new{four} of them which are useful for Section \ref{contributions_bernstein}. For a more exhaustive introduction to Bernstein expansions, as well as proofs of the basic properties, we refer the interested reader to \cite{smithThesis}.
\begin{property}[{Cardinality \cite[(3.14)]{smithThesis}}]
\label{card_prop}
The number of Bernstein coefficients in the Bernstein expansion of multi-degree $\kb$ is equal to $ (\kb+\bone)^{\bone} = \prod_{i = 1 }^n(k_i+1).$
\end{property}
\new{\begin{property}[{Linearity \cite[(3.2.3)]{smithThesis}}]
\label{lin_prop}
Given two polynomials $f_1$ and $f_2$, one has:
$$ b_{\ab}^{(c f_1+f_2)} = c b_{\ab}^{(f_1)} + b_{\ab}^{(f_2)} ~, \quad \forall c \in \R,$$
where the above Bernstein expansions are of the same multi-degree.
\end{property}}
%
%
\begin{property}[{Enclosure \cite[(3.2.4)]{smithThesis}}]
\label{enclosure_prop}
The minimum (resp.~maximum) of a polynomial $f$ over $[0, 1]^n$ can be lower bounded (resp.~upper bounded) by the minimum (resp.~maximum) of its Bernstein coefficients:
\new{$$ \min_{\ab \leq \kb} b_{\ab}  \leq f(\xb) \leq \max_{\ab \leq \kb} b_{\ab},~~ \forall \xb\in[0,1]^n .$$}
\end{property}
\begin{property}[{Sharpness \cite[(3.2.5)]{smithThesis}}]
\label{sharp_prop}
If the minimum (resp.~maximum) of the $b_{\ab}$ is reached at $\ab$ coinciding with a corner of the box $[0,k_1]\times\dots\times[0,k_n]$, then $b_{\ab}$ is the minimum (resp.~maximum) of $f$ over $[0,1]^n$.
\end{property}

Property \ref{card_prop} gives the maximal computational cost needed to find a lower bound of $\underline{f}$ for a Bernstein expansion of fixed multi-degree $\kb$. Property \ref{enclosure_prop} is used to bound from below optimal values, while Property \ref{sharp_prop} allows determining in some cases if the lower bound is optimal. 
{\subsection{Bounds of Rational Functions with Bernstein Expansions}
\label{preliminaries_bernstein_rational}
In this section we recall how to obtain bounds for the range of multivariate rational functions by using Bernstein expansions. The following result can be found in~\cite[Theorem~3.1]{NGSM12}.
\begin{theorem}
\label{th:bernstein_rational}
Let $f_1, f_2 \in \R[\xb]$ of respective multi-degrees $\db_1$ and $\db_2$. 
Given $\kb \geq \max \{\db_1, \db_2\}$, let us denote by $b_{\ab}^{(f_1)}$ and $b_{\ab}^{(f_2)}$ the Bernstein coefficients of multi-degree $\kb$ for $f_1$ and $f_2$, respectively. 
Let us assume that $f_2$ is positive over $[0, 1]^n$ and that all Bernstein coefficients $b_{\ab}^{(f_2)}$ are positive. Then, for $f := \frac{f_1}{f_2}$, one has
\begin{align}
\label{eq:bernstein_rational}
\min_{\ab \leq \kb} \frac{b_{\ab}^{(f_1)}}{b_{\ab}^{(f_2)}} \leq f(\xb) \leq \max_{\ab \leq \kb} \frac{b_{\ab}^{(f_1)}}{b_{\ab}^{(f_2)}} ,~~ \forall \xb\in[0,1]^n \,.
\end{align}
\end{theorem}
Note that to use the bounds from Theorem~\ref{th:bernstein_rational}, one first has to compute the Bernstein coefficients of $f_2$ for sufficiently large multi-degree $\kb$ in order to ensure that the denominators in~\eqref{eq:bernstein_rational} do not vanish. We emphasize that our assumption on the sign of $f_2$ is equivalent to suppose that $f_2(\xb) \neq 0$ for all $\xb \in [0, 1]^n$. 
\if{
We will recall later the linear convergence of Bernstein expansions w.r.t.~multi-degree elevation when approximating the range of rational functions (see e.g.~\cite[Theorem~5]{ GarloffBern16}).
\begin{theorem}
\label{th:cvgbern}

\end{theorem}
}\fi}
\subsection{Dense and Sparse Krivine-Stengle Representations}
\label{preliminaries_handelman}
In this section, we first give the necessary background on Krivine-Stengle representations, used in the context of polynomial optimization. Then, we present a sparse version based on \cite{schweighofer06}. These notions are applied later in Section \ref{contributions_handelman}.
\subsubsection{Dense Krivine-Stengle representations}
\label{sec:denseKS}
Krivine-Stengle certificates for positive polynomials can first be found in \cite{krivineanneaux,stengle} (see also \cite[Theorem 1(b)]{bsos}). Such certificates give representations of positive polynomials over a compact set $\K := \{\xb\in\mathbb{R}^n : 0\leq g_i(\xb)\leq 1,~i=1,\dots,p\}$, with $g_1, \dots, g_p \in \R[\x]$. 
\new{We denote $\dg = \max_i(\text{deg}(g_i))$.} 
\new{The compact set $\K$ is called a {\em basic semialgebraic set}, that is a set defined by a conjunction of finitely many polynomial inequalities. In the sequel, we assume that $\K \subseteq [0, 1]^n$ and that 
$x_i$ ($i=1,\dots,n)$ are among the polynomials $g_j$ in the  definition of $\K$.
This implies that the family $\{1,g_i\}_{i\leq p}$ generates $\R[\x]$ as an $\R$-algebra, which is a necessary assumption for Theorem~\ref{dense_stengle_th}.
}\\
Given $\ab = (\alpha_1,\dots,\alpha_p)$ and $\bb=(\beta_1,\dots,\beta_p)$, let us define the polynomial $h_{\ab,\bb}(\xb) = \g^{\ab} (\mathbf{1} - \g)^{\bb} = \prod_{i=1}^p g_i^{\alpha_i}(1-g_i)^{\beta_i}$. \newline
For instance on the two-dimensional unit box, one has $n=p=2$, $\cSet = [0,1]^2 = \{\xb\in\mathbb{R}^2 : 0\leq x_1\leq 1 \,, \ 0\leq x_2\leq 1\}$. With $\ab = (2,1)$ and $\bb = (1,3)$, one has $h_{\ab,\bb}(\xb) = x_1^2 x_2 (1 - x_1) (1 - x_2)^3$.
\begin{theorem}[Dense Krivine-Stengle representations]
\label{dense_stengle_th}
Let $\pfunc \in \R[\x]$ be a positive polynomial over $\K$. Then there exist $k \in \N$ and a finite number of nonnegative weights $\lambda_{\ab,\bb}\geq 0$ such that:
	\begin{equation}
	\label{eq:ks}
		\pfunc(\xb) = \sum_{|\ab+\bb|\leq k}\lambda_{\ab,\bb}h_{\ab,\bb}(\xb), \quad \forall\xb \in \mathbb{R}^n.
	\end{equation}
\end{theorem}	
\new{We denote by $\mathcal{H}_k(\X)$ the set of polynomials having a dense Krivine-Stengle representation (of degree at most $k$) as in~\eqref{eq:ks}.}
It is possible to compute the weights $\lambda_{\ab,\bb}$ by identifying in the monomial basis the coefficients of the polynomials in the left and right sides of~\eqref{eq:ks}. 
\new{Given $\pfunc \in \mathcal{H}_k(\X)$ and} 
denoting by $(\pfunc)_{\gb}$ the monomial coefficients of $\pfunc \in \mathcal{H}_k(\X)$, \new{with $\gb\in\mathbb{N}_{\kh}^n := \{\gb\in\mathbb{N}^n : |\gb|\leq \kh=k \, \dg\}$}, the $\lambda_{\ab,\bb}$ fulfill the following equalities:
\begin{equation}
\pfunc_{\gb} = \sum_{|\ab+\bb|\leq k}\lambda_{\ab,\bb}(h_{\ab,\bb})_{\gb}, \quad \forall \gb\in\mathbb{N}_{\kh}^n.
\end{equation}
\subsubsection{Global optimization using the dense Krivine-Stengle representations}
\label{sec:globalKS}
Here we consider the polynomial minimization problem  $\underline{f} : = \min_{\xb\in \cSet }f(\xb)$, with $f$ a polynomial of degree $d$. We can rewrite this problem as the following infinite dimensional problem:
\new{\begin{equation}
	\label{infopti_eq}
     \begin{aligned}
        \underline{f} :=& \max\limits_{t\in\mathbb{R}}~~t,\\
        	  & \text{s.t.}~~f(\xb) - t \geq 0 \,, \quad \forall \xb\in  \cSet .  \\   
      \end{aligned}
\end{equation}}
The idea is to look for a hierarchy of finite dimensional linear programming (LP) relaxations by using Krivine-Stengle representations of the positive polynomial $\pfunc = f - t$ involved in Problem~\eqref{infopti_eq}. 
%
%
Applying Theorem~\ref{dense_stengle_th} to this polynomial, we obtain the following LP problem for each $k \geq d$:
\new{\begin{equation}
	\label{eq:pk}
	     \begin{aligned}
        	  \underline{f}_k := \max\limits_{t,\lambda_{\ab,\bb}} 
        	  \quad & t,\\
        	  \text{s.t } \quad & (f - t)_{\gb} = \sum_{|\ab+\bb|\leq k}\lambda_{\ab,\bb}(h_{\ab,\bb})_{\gb} \,, \quad \forall \gb\in\mathbb{N}_{\kh}^n \,,\\   
      		  \quad & \lambda_{\ab,\bb}\geq 0.
      	   \end{aligned}
	\end{equation} }

\new{Note that $\underline{f}_k = \max \{t : f - t \in \mathcal{H}_k(\X) \}$.}
As in~\cite[(4)]{bsos}, one has:
\begin{theorem}[Dense Krivine-Stengle LP relaxations]
\label{denseLPrelax_th}
The sequence of optimal values $(\underline{f}_k)$ satisfies  $\underline{f}_k \rightarrow \underline{f}$ as $k\rightarrow +\infty$. Moreover each $\underline{f}_k$ is a lower bound of $\underline{f}$.
\end{theorem} 
At fixed $k$, the total number of variables of LP~\eqref{eq:pk} is given by the number of $\lambda_{\ab,\bb}$ and $t$, that is \new{$\binom{2p+k}{k} +1$, where $p$ is the dimension of $\g$. The number of constraints is equal to the cardinality of $\mathbb{N}_{\kh}^n$, which is $\binom{n+\kh}{\kh}$. We recall that $\kh = k \, \dg$.}
\new{In the particular case where $\K$ is an hypercube, the LP has $\binom{2n+k}{k} +1$ variables and $\binom{n+k}{k}$ constraints.}
\subsubsection{Sparse Krivine-Stengle representations}
\label{sec:sparseKS}
We now explain how to derive less computationally expensive LP relaxations, by relying on sparse Krivine-Stengle representations. 
 For $I \subseteq \{1,\dots,n\}$, let $\mathbb{R}[\xb, I]$ be the ring of polynomials restricted to the variables $\{x_i~:~i\in I \}$. 
We borrow the notion of a sparsity pattern from \cite[Assumption 1]{sbsos}:
\begin{definition}[Sparsity Pattern]
\label{sparse_def}
Given $m\in\mathbb{N}$, $I_j\subseteq\{1,\dots,n\}$, and $J_j \subseteq\{1,\dots,p\}$ for all $j = 1,\dots,m$, a sparsity pattern is defined by the four following conditions:
\begin{itemize}
\item $f$ can be written as: $f = \sum_{j=1}^{m}f_j$ with $f_j\in\mathbb{R}[\xb,I_j]$, 
\item $g_i \in\mathbb{R}[\xb,I_j]$ for all $i \in J_j$, for all $j = 1, \dots, m$,
\item $\bigcup_{j=1}^m I_j = \{1,\dots,n\}$ and $\bigcup_{j=1}^m J_j = \{1,\dots,p\}$,
\item (Running Intersection Property)~for all $j=1,\dots,m-1$, there exists $s\leq j$ s.t. $I_{j+1} \cap \bigcup_{i=1}^j I_i \subseteq I_s$.
\end{itemize}
\end{definition}
As an example, the four conditions stated in Definition~\ref{sparse_def} are satisfied while considering $f(\xb) = x_1 x_2 + x_1^2 x_3$ on the hypercube $\cSet = [0, 1]^3$. Indeed, one has $f_1(\xb) = x_1 x_2 \in \mathbb{R}[\xb,I_1]$, $f_2(\xb) = x_1^2 x_3 \in \mathbb{R}[\xb,I_2]$ with $I_1 = \{1,2\}$, $I_2 = \{1,3\}$. Taking $J_1 = I_1$ and $J_2 = I_2$, one has $g_i = x_i \in \mathbb{R}[\xb,I_j]$ for all $i \in I_j$, $j=1,2$.

Let us consider a given sparsity pattern as stated above. By noting $n_j = |I_j|$, $p_j = |J_j|$, then the set $ \cSet  = \{\xb\in\mathbb{R}^n : 0\leq g_i(\xb)\leq 1, \, i=1,\dots,p\}$ yields subsets $ \cSet_j = \{\xb\in\mathbb{R}^{n_j} : 0\leq g_i(\xb)\leq 1,~i\in J_j\}$, with $j = 1,\dots,m$. If $ \cSet $ is a compact subset of $\mathbb{R}^n$ then each $\cSet_j$ is a compact subset of $\mathbb{R}^{n_j}$.
As in the dense case, let us note $h_{\ab_j,\bb_j} :=  \mathbf{g}^{\ab_j}(\mathbf{1}- \mathbf{g})^{\bb_j}$, for given $\ab_j, \bb_j \in \mathbb{N}^{n_j}$. \newline

The following result, a sparse variant of Theorem~\ref{dense_stengle_th}, can be retrieved from~\cite[Theorem 1]{sbsos} but we also provide here a shorter alternative proof by using~\cite{schweighofer06}.
\begin{theorem}[Sparse Krivine-Stengle representations]
\label{spkrivrep_th}
Let $f,g_1,\dots,g_p\in\mathbb{R}[\xb]$ be given and assume that there exist $I_j$ and $J_j$, $j=1,\dots,m$, which satisfy the four conditions stated in Definition~\ref{sparse_def}. If $f$ is positive over $\mathbf{K}$, then there exist $\phi_j \in \mathbb{R}[\xb,I_j]$, $j=1,\dots,m$ such that $f=\sum_{j=1}^m\phi_j$ and $\phi_j > 0$ over $\cSet_j$. In addition, there exist $k\in\mathbb{N}$ and finitely many nonnegative weights $\lambda_{\ab_j,\bb_j}$, $j=1,\dots,m$, such that:
\begin{equation}
\phi_j = \sum_{|\ab_j+\bb_j|\leq k}\lambda_{\ab_j,\bb_j} h_{\ab_j,\bb_j}~, \quad j=1,\dots,m.
\end{equation}
\end{theorem}
\begin{proof}
From \cite[Lemma 3]{schweighofer06}, \new{there exist} $\phi_j \in \mathbb{R}[\xb,I_j]$, \new{$j=1,\dots, m$}, such that $f=\sum_{j=1}^m\phi_j$ and $\phi_j>0$ on $\mathbf{K}_j$. 
Applying Theorem~\ref{dense_stengle_th} on each $\phi_j$, there exist $k_j\in\mathbb{N}$ and finitely many nonnegative weights $\lambda_{\ab_j,\bb_j}$ such that $\phi_j = \sum_{|\ab_j+\bb_j|\leq k_j}\lambda_{\ab_j,\bb_j}h_{\ab_j,\bb_j}.$
By taking $k = \max_{1 \leq j \leq m} \{k_j\}$, we complete the representations with as many zero $\lambda$ as necessary to obtain the desired result.
\end{proof}
\new{As in Section~\ref{sec:denseKS}, we note $\mathcal{H}_k(\K)$ the set of functions with sparse Krivine-Stengle representations (of degree at most $k$) given in Theorem~\ref{spkrivrep_th}.}

%
In Theorem~\ref{spkrivrep_th}, one assumes that $f$ can be written as the sum $f = \sum_{j=1}^m f_j$, where each $f_j$ is not necessarily positive. The first result of the theorem states that $f$ can be written as another sum $f = \sum_{j=1}^m \phi^j$, where each $\phi_j$ is now positive.
%
As in the dense case, the $\lambda_{\ab_j,\bb_j}$ can be computed by equalizing the coefficients in the monomial basis. 
\if{
Let us note $\Ic_j^k = \{\gb\in\mathbb{N}_k^n~:~\gamma_i=0\textnormal{ if } i \notin I_j \}$. Then the $\lambda_{\ab_j,\bb_j}$ fulfill the following equalities:
\begin{equation}
\label{sparseequal_eq}
     \begin{aligned}
        	  (f(\xb))_{\gb} = & \sum_{j:\gb\in\Ic_j^k}(\phi_j)_{\gb} ~, \quad \forall\gb\in\bigcup_{j=1}^{m}\Ic_j^k~,\\   
      		  (\phi_j)_{\gb} = & \sum_{|\ab_j+\bb_j|\leq k}\lambda_{\ab_j,\bb_j}(h_{\ab_j,\bb_j})_{\gb} \,, \quad j=1,\dots,m .
     \end{aligned}
\end{equation}
}\fi
We also obtain a hierarchy of LP relaxations to approximate the solution of polynomial optimization problems. For the sake of conciseness, we only provide these relaxations as well as their computational costs in the particular context of roundoff error bounds in Section~\ref{contributions_handelman}.


\section{Two new \new{algorithms} to compute roundoff error bounds}
\label{contributions}
This section is dedicated to our main contributions. We provide two new \new{algorithms} to compute absolute roundoff error bounds using either Bernstein expansions or sparse Krivine-Stengle representations. 
\new{The first algorithm, denoted by~$\fpbern$, \new{takes as input} a program implementing the expression of a rational function $f$, with variables $\x$ satisfying input constraints encoded by the product $\X$ of closed intervals. After adequate change of variables, we assume without loss of generality that $\X := [0, 1]^n$.}
\new{The second algorithm, denoted by~$\fpkristen$, \new{takes as input} a program implementing a polynomial expression $f$, with variables $\x$ satisfying input constraints encoded by a basic compact semialgebraic set $\X \subseteq [0, 1]^n$, as in Section~\ref{sec:denseKS}.
}

Following the simple rounding model described in Section~\ref{preliminaries_floating_point}, we denote by $\hat{f}(\xb,\eb)$ the rounded expression of $f$ after introduction of the rounding variables $\eb$ (one additional variable is introduced for each real variable $x_i$ or constant as well as for each arithmetic operation $+$,$\times$, $-$ or \new{$/$}). For a given machine epsilon $\varepsilon$, these error variables also satisfy a set of constraints encoded by the box $[-\varepsilon, \varepsilon]^m$.\\
As explained in~\cite[Section 3.1]{real2float}, we can decompose the  roundoff error as follows: $r(\xb, \eb) := \hat{f}(\xb,\eb) - f(\xb) = l(\xb,\eb) + h(\xb, \eb)$, where $l(\xb,\eb) := \sum_{j=1}^m \frac{\partial r(\xb,\eb)} {\partial e_j} (\xb,0)  e_j  = \sum_{j=1}^m s_j (\xb) e_j$. 
One obtains an enclosure of $h$ using interval arithmetic to bound second-order error terms in the Taylor expansion of $r$ w.r.t.~$\eb$ (as in~\cite{fptaylor,real2float}).\newline 
%
\new{
Let us note $I^l := [\underline{l}, \overline{l}]$ the interval enclosure of $l$, with $\underline{l} := \min_{(\xb,\eb) \in \X \times \E} l(\xb,\eb)$ and $\overline{l} := \max_{(\xb,\eb) \in \X \times \E} l(\xb,\eb)$. For each $\e \in \E$, we also define $l_{\e}(\x) := l(\x, \e)$ on $\X$.
}
After dividing each error variable $e_j$ by $\varepsilon$, we now consider the optimization of the (scaled) linear part $l' := l/ \varepsilon$ of the roundoff error. 

\subsection{Bernstein expansions of roundoff errors}
\label{contributions_bernstein}
\new{
The first method is the approximation of $\underline{l}'$ (resp.~$\overline{l}'$) by using Bernstein expansions of the polynomials involved in $l'$. 
\subsubsection{Polynomial expressions} 
We start with the simpler case where $l'$ is the scaled linear part of the roundoff error of a program implementing a polynomial expression $f$.
Let $\db$ be the  multi-degree of $f$.
Note that $\db$ is greater than the multi-degree of $s_j(\x) := \frac{\partial r(\x,\e)} {\partial e_j} (\x,0)$, appearing in the definition of $l'$, for all $j=1,\dots,m$.
}
Our procedure is based on the following proposition:
%
%
\begin{proposition}
\label{bernth}
For each $\kb \geq \db$, the polynomial $l'$ can be bounded as follows:
\begin{equation}
\underline{l}'_{\kb}  \leq l'(\xb,\eb) \leq \overline{l}'_{\kb} \,,  \quad \forall (\xb,\eb) \in \X \times \E \,,
\end{equation} 
\new{with $\overline{l}'_{\kb} = \max_{\ab \leq \kb}\sum_{j = 1}^m |b_{\ab}^{(s_j)}|$ and $\underline{l}'_{\kb} = - \overline{l}'_{\kb}$.}
\end{proposition}
\begin{proof}
We write $l'_{\eb}\in\mathbb{R}[\xb]$ the polynomial $l'(\xb,\eb)$ for a given $\eb\in \mathbf{E}$. Property~\ref{enclosure_prop} provides the enclosure of $l'_{\eb}(\xb)$ w.r.t. $\xb$ for a given $\eb\in \mathbf{E}$:
\begin{equation}
\label{eq:enclosurele}
\min_{\ab \leq \kb} b_{\ab}^{(l'_{\eb})} \leq l'_{\eb}(\xb)\leq \max_{\ab \leq \kb} b_{\ab}^{(l'_{\eb})} \,, \quad \forall \xb\in[0,1]^n \,,
\end{equation} 
where each Bernstein coefficient satisfies $b_{\ab}^{(l'_{\eb})} = \sum_{j=1}^m e_j b_{\ab}^{(s_j)}$ by Property \ref{lin_prop} (each $e_j$ being a scalar in $[-1,1]$). 
The proof of  the left inequality comes from:
\begin{align*}
\min_{\eb\in[-1,1]^m}\bigl(\min_{\ab \leq \kb}(\sum_{j=1}^m e_j b_{\ab}^{(s_j)})\bigr) &
= \min_{\ab \leq \kb}\bigl( \min\limits_{\eb\in[-1,1]^m}(\sum\limits_{j=1}^m e_j b_{\ab}^{(s_j)})\bigr)\\
  &= \min\limits_{\ab \leq \kb} \sum\limits_{j=1}^m -|b_{\ab}^{(s_j)}| \\
&  = - \max\limits_{\ab \leq \kb} \sum\limits_{j=1}^m |b_{\ab}^{(s_j)}|~~.
\end{align*}
The proof of the right inequality is similar.
\end{proof}
\begin{remark}
\label{rk:cost1}
\new{By Property~\ref{card_prop}}, the computational cost of $\underline{l}'_{\kb}$ is  $m (\kb + \bone)^{\bone}$  since we need to compute the Bernstein coefficients for each $s_j(\xb)$. This cost is polynomial in the degree and exponential in $n$ but is linear in $m$.
\end{remark}
\new{
\subsubsection{Rational function expressions}
We now consider the more general case where $l'$ is the scaled linear part of the roundoff error of a program implementing a rational expression $f := \frac{f_1}{f_2}$ with $f_2 (\x) > 0$ for all $\x \in \X = [0, 1]^n$.  
Let $\db_1$ and $\db_2$ be the multi-degrees of $f_1$ and $f_2$ and $\db := \max\{\db_1,\db_2 \}$. For all $j=1,\dots,m$, we can write each $s_j$, appearing in the definition of $l'$, as $s_j := \frac{\partial r(\x,\e)} {\partial e_j} (\x,0) = \frac{p_j(\x)}{q_j(\x)^2}$, with $p_j$ and $q_j^2$ of multi-degrees less than $2 \db$, and $q_j (\x) \neq 0$ for all $\x \in \X = [0, 1]^n$.
\newline
We extend Proposition~\ref{bernth} as follows.}
\new{
\begin{proposition}
\label{bernthrat}
For each $\kb \geq 2\db$, the rational function $l'$ can be bounded as follows:
\begin{equation}
\underline{l}'_{\kb}  \leq l'(\xb,\eb) \leq \overline{l}'_{\kb} \,,  \quad \forall (\xb,\eb) \in \X \times \E \,,
\end{equation} 
with $\overline{l}'_{\kb} = \displaystyle\max_{\ab \leq \kb}\sum_{j = 1}^m \frac{|b_{\ab}^{(p_j)}|}{|b_{\ab}^{(q_j^2)}|} $ and $\underline{l}'_{\kb} = - \overline{l}'_{\kb}$.
\end{proposition}
}
\new{
\begin{proof}
We first handle the case when for some $j \in \{1,\dots,m\}$, there exists $\ab \leq \kb$ such that $b_{\ab}^{(q_j^2)} = 0$. In this case, one has  $\overline{l}'_{\kb} = \infty$ and $\underline{l}'_{\kb} = - \infty$, so both inequalities trivially hold. \newline
Next, let us assume that all considered Bernstein coefficients of $q_j^2$ are positive. By Theorem~\ref{th:bernstein_rational}, one has for all $j=1,\dots,m$:
\[
\min_{\ab \leq \kb} \frac{b_{\ab}^{(p_j)}}{b_{\ab}^{(q_j^2)}}  \leq \frac{p_j(\x)}{q_j(\x)^2} \leq \max_{\ab \leq \kb} \frac{b_{\ab}^{(p_j)}}{b_{\ab}^{(q_j^2)}} \,, \quad \forall \xb\in[0,1]^n \,, 
\]
yielding 
\[
- \max_{\ab \leq \kb} \frac{|b_{\ab}^{(p_j)}|}{|b_{\ab}^{(q_j^2)}|}  \leq \frac{p_j(\x)}{q_j(\x)^2} \, e_j \leq \max_{\ab \leq \kb} \frac{|b_{\ab}^{(p_j)}|}{|b_{\ab}^{(q_j^2)}|} \,, \quad \forall \xb\in[0,1]^n\,,
\]
which implies the desired result.
\end{proof}
}
\new{
\subsubsection{Algorithm~$\fpbern$}
The algorithm~$\fpbern$, stated in Figure~\ref{alg:fpbern} takes as input $\x := (x_1,\dots,x_n)$, the set $\X := [0, 1]^n$ of bound constraints over $\x$, a rational function expression $f$, the corresponding rounded expression $\hat{f}$ with rounding variable $\e := (e_1,\dots,e_m)$ and the set of bound constraints $\E := [-\varepsilon, \varepsilon]^m$ over $\e$. 
From Line~\lineref{line:r} to Line~\lineref{line:iabound}, the algorithm $\fpbern$ is implemented exactly as in $\fptaylor$~\cite{fptaylor} as well as in~$\realtofloat$~\cite{real2float}. The absolute roundoff error $r$ is decomposed as the sum of an expression $l$ and a remainder $h$ and the enclosure $I^h$ of $h$ is computed thanks to a subroutine $\iabound$ performing basic interval arithmetics. The main difference between $\fpbern$ and $\fptaylor$ (resp.~$\realtofloat$) is that the enclosure of $l':= \frac{l}{\varepsilon}$ is obtained in Line~\lineref{line:blbelow} thanks the computation of the Bernstein expansion as in Proposition~\ref{bernthrat}.
}

\begin{figure}[!t]
\new{
\begin{algorithmic}[1]                    
\Require input variables $\x$, input constraints $\X = [0, 1]^n$, rational function expression $f := \frac{f_1}{f_2}$ with $f_1$, $f_2$ of respective multi-degrees $\db_1$ and $\db_2$, rounded expression $\hat{f}$, error variables $\e$, error constraints $\E = [-\varepsilon, \varepsilon]^m$, multi-degree $\kb \geq 2 \max\{\db_1,\db_2\}$
\Ensure interval enclosure $I_{\kb}$ of the error $\hat{f} - f$ over $\K := \X \times \E$
\State  $r(\x, \e) := \hat{f}(\x,\e) - f(\x)$ \label{line:r}
\For {$j \in \{1,\dots,m\}$} Compute the polynomials $p_j$ and $q_j$ such that $s_j(\x) \gets \frac{\partial r(\x,\e)} {\partial e_j} (\x,0) = \frac{p_j(\x)}{q_j(\x)^2}$
\EndFor
\State $l(\x,\e) \gets \sum_{j=1}^m \frac{p_j(\x)}{q_j(\x)^2} \, e_j$ , $l' \gets \frac{l}{\varepsilon}$ \label{line:l}
\State $h := r - l$ \label{line:h}
\State $I^h := \iaboundfun{h}{\K}$ \label{line:iabound}
\State $\overline{l}'_{\kb} \gets \displaystyle\max_{\ab \leq \kb}\sum_{j = 1}^m \frac{|b_{\ab}^{(p_j)}|}{|b_{\ab}^{(q_j^2)}|}$ \label{line:blbelow}
\State $I_{\kb}^l \gets [\varepsilon \underline{l}'_{\kb}, \varepsilon \overline{l}'_{\kb}]$ \label{line:fpbernbound}
\State \Return $I_{\kb} := I_{\kb}^l + I^h$ 
\end{algorithmic}
}
\caption{\new{$\fpbern$: our algorithm to compute roundoff errors bounds of programs implementing rational function expressions with Bernstein expansions.}}
\label{alg:fpbern}
\end{figure}
\new{
Later on, we provide the convergence rate of Algorithm $\fpbern$ in Section~\ref{complexity} when $f$ is a polynomial expression.
\if{
The following result states that we can approximate $I^l$ as close as desired while relying on Algorithm~$\fpbern$.
\begin{proposition}[Convergence of Algorithm~$\fpbern$]
\label{th:fpbern}
Let $f := \frac{f_1}{f_2}$, with $f_1, f_2 \in \R[\xb]$ of respective multi-degrees $\db_1,\db_2$ and $f_2(\x) > 0$ for all $\x \in [0, 1]^n$. Let $\db := \max \{\db_1,\db_2 \}$.
After running Algorithm~$\fpbern$ on $f$ and $\kb \geq 2 \db$, let $I_{\kb}^l$ be the interval enclosure returned  at Line~\lineref{line:fpbernbound}. Then, the sequence $(I_{\kb}^l)_{\kb \geq 2 \db}$ converges to $I^l$.
\end{proposition}
\begin{proof}
Given a multivariate rational function, the result from~\cite[Theorem~5]{GarloffBern16} states that one has convergence of the bounds given by Bernstein expansions to the range of the rational function. This convergence is linear with respect to degree elevation $k$ of the Bernstein expansion of multi-degree $\kb = (k,\dots,k)$, yielding the desired result.
\end{proof}
}\fi
}
\begin{example}
\label{bernex}
For the polynomial $l$ defined in~\eqref{eq:l} (Section~\ref{overview}),  \new{one has $l(x,\eb) = (2x^2-x) e_1 + x^2 e_2 + (x^2-x) e_3$. }Applying the above method with $\kb = \db = 2$, one considers the following Bernstein coefficients:
\[ 
b_0^{(l'_{\eb})} = 0, \quad b_1^{(l'_{\eb})} = -\frac{e_1}{2}-\frac{e_3}{2}, \quad b_2^{(l'_{\eb})} = e_1+e_2.
\]
The number of Bernstein coefficients w.r.t.~$x$ is $3$, which is much lower than the one w.r.t.~$(x,\eb)$, which is equal to $24$. 
One can obtain an upper bound (resp.~lower bound)  by taking the maximum (resp.~minimum) of the Bernstein coefficients. 
\new{In this case, $\max_{\eb\in [-1,1]^3} b_1^{(l'_{\eb})} = 0$, $\max_{\eb\in [-1,1]^3} b_2^{(l'_{\eb})} = 1$ and $\max_{\eb\in [-1,1]^3} b_3^{(l'_{\eb})} = 2$. Thus, one obtains $\overline{l}'_{\kb} = 2$ as an upper bound of  $\overline{l}'$.}
\end{example}
\subsection{Sparse Krivine-Stengle representations of roundoff errors}
\label{contributions_handelman}
\new{Here we assume that $f$ is a polynomial and $\X \subseteq [0, 1]^n$ is a basic compact semialgebraic set. We note $d$ the degree of $f$. We also note $\gX$ the vector of $p$ polynomial constraints whose conjunction defines $\X$.}
We explain how to compute lower bounds of $\underline{l}' := \min_{(\xb,\eb) \in \mathbf{X} \times \mathbf{E}} l'(\xb,\eb)$ by using sparse Krivine-Stengle representations. 
We obtain upper bounds of $\overline{l}' := \max_{(\xb,\eb) \in \mathbf{X} \times \mathbf{E}} l'(\xb,\eb)$ in a similar way. Note that the degree of $l'$ is less than $d+1$.

For the sake of consistency with Section~\ref{preliminaries_handelman}, we introduce the variable $\yb \in \mathbb{R}^{n+m}$ defined by \new{$y_i := x_i$, $i=1,\dots,n$ and $y_i := e_{i-n}$, $i=n+1,\dots,n+m$.
Then, one can write the set $\cSet = \mathbf{X}\times\mathbf{E}$  as follows:
 \begin{equation}
 \cSet = \{\yb\in\mathbb{R}^{n+m} : 0 \leq g_j(\yb) \leq 1 \,, \quad j=1,\dots, p+m \} \,,
  \end{equation}
with  $g_j(\yb) := g^{\X}_j(\x)$, for each $j=1,\dots,p$ and $g_j(\yb) := \frac{1}{2}+\frac{e_j}{2}$, for each $j=p+1,\dots,p+m$. }
\begin{lemma}
\label{pattern}
For each $j=1, \dots, m$, let us define $I_j := \{1,\dots,n,n+j\}$ and \new{ $J_j := \{1,\dots,p,p+j\}$. }Then the sets $I_j$ and $J_j$ satisfy the four conditions stated in~Definition~\ref{sparse_def}.
\end{lemma}
\begin{proof}
The first condition holds as $l'(\yb) = l'(\xb,\eb) = \sum_{j=1}^m s_j(\xb, \eb) e_j =  \sum_{j=1}^m s_j(\yb)  e_j$, with $s_j(\yb)\in\mathbb{R}[\yb,I_j]$. The second and third condition are  obvious. The running intersection property  comes from $I_{j+1} \cap I_j = \{1, \dots, n \} \subseteq I_j$. 
\end{proof}
Given \new{$\ab,\bb \in \mathbb{N}^{p+1}$, one can write $\ab = (\ab',\gamma)$ and $\bb = (\bb',\delta)$, for $\ab',\bb' \in \mathbb{N}^{p}$, $\gamma,\delta \in \mathbb{N}$.}
In our case, this gives the following formulation for the polynomial $h_{\ab_j,\bb_j}(\yb)= 
\mathbf{g}^{\ab_j}
( \bone - \mathbf{g})^{\bb_j}$:
\new{
\begin{align*}
h_{\ab_j,\bb_j}(\yb) & = h_{\ab_j',\bb_j',\gamma_j,\delta_j}(\xb,\eb)  \\
& = \gX(\xb)^{\ab_j'}
( \bone - \gX(\xb))^{\bb_j'}
( \frac{1}{2}+\frac{e_j}{2})^{\gamma_j}
( \frac{1}{2}-\frac{e_j}{2})^{\delta_j}
 \,. 
 \end{align*}
 }
For instance, with the polynomial $l'$ considered in Section~\ref{overview} \new{on the interval $[0,1]$ }and depending on $x, e_1, e_2, e_3$, one can consider the multi-indices $\ab_1 = (1,2)$, $\bb_1 = (2,3)$ associated to the roundoff variable $e_1$. Then  $h_{\ab_1,\bb_1}(\yb) = x (1 - x)^2 (\frac{1}{2}+\frac{e_1}{2})^2 ( \frac{1}{2}-\frac{e_1}{2})^{3}$.
 
Now, we consider the following hierarchy of LP relaxations, for each $k \geq d+1$:
\begin{align}
\underline{l}'_k := \max_{t,\lambda_{\ab_j,\bb_j}} \quad & t  \,, \nonumber \\
\text{s.t } \quad & l'-t  = \sum_{j=1}^m\phi_j \,, \label{sparsetheq} \\   
\quad & \phi_j  =  \sum_{|\ab_j+\bb_j|\leq k}\lambda_{\ab_j,\bb_j} h_{\ab_j,\bb_j} \,, \ j=1,\dots,m \,, \nonumber \\
\quad & \lambda_{\ab_j,\bb_j}\geq 0  \,, \quad j=1,\dots,m \nonumber \,.
\end{align}
\if{
\begin{equation}
\label{sparsetheq}
\begin{aligned}
\underline{l}'_k := \max_{t,\lambda_{\ab_j,\bb_j}} \quad & t  \,,  \\
\text{s.t } \quad & l'-t  = \sum_{j=1}^m\phi_j \,,  \\   
\quad & \phi_j  =  \sum_{|\ab_j+\bb_j|\leq k}\lambda_{\ab_j,\bb_j} h_{\ab_j,\bb_j} \,, \quad j=1,\dots,m \,, \\
\quad & \lambda_{\ab_j,\bb_j}\geq 0  \,, \quad j=1,\dots,m \,.
\end{aligned}
\end{equation}
}\fi
\new{Note that $\underline{l}'_k = \max \{t : l' - t \in \mathcal{H}_k(\K) \}$, where $\mathcal{H}_k(\K)$ is the set of sparse Krivine-Stengle representations defined in Section~\ref{sec:sparseKS}. Similarly, we obtain $\overline{l}'_k$ while replacing $\max$ by $\min$ and $l'-t$ by $t - l'$ in LP~\eqref{sparsetheq}, that is $\overline{l}'_k = \min \{t : t - l' \in \mathcal{H}_k(\K) \}$.
}

\new{
The algorithm~$\fpkristen$ stated in Figure~\ref{alg:fpkristen} is very similar to $\fpbern$. By contrast with $\fpbern$, $\fpkristen$ takes as input a polynomial $f$ of degree $d$ and does not work for programs implementing rational functions. However, $\fpkristen$ can handle a general basic compact semialgebraic input set of constraints $\X$, i.e.~a set $\X$ defined by a finite conjunction of polynomial inequalities. The lower (resp.~upper) bound of $l':= \frac{l}{\varepsilon}$ is obtained in Line~\lineref{line:kslbelow} (resp.~Line~\lineref{line:kslabove}) by solving the LP relaxation~\eqref{sparsetheq} at order $k \geq d+1$.
}
\begin{figure}[!t]
\new{
\begin{algorithmic}[1]                    
\Require input variables $\x$, input constraints $\X$, polynomial expression $f$ of degree $d$, rounded expression $\hat{f}$, error variables $\e$, error constraints $\E = [-\varepsilon, \varepsilon]^m$, degree $k \geq d+1$
\Ensure interval enclosure $I_k$ of the error $\hat{f} - f$ over $\K := \X \times \E$
\State  $r(\x, \e) := \hat{f}(\x,\e) - f(\x)$ \label{line:rbis}
\For {$j \in \{1,\dots,m\}$} $s_j(\x) \gets \frac{\partial r(\x,\e)} {\partial e_j} (\x,0)$
\EndFor
\State $l(\x,\e) \gets r(\x, 0) + \sum_{j=1}^m s_j(\x) \, e_j$ , $l' \gets \frac{l}{\varepsilon}$ \label{line:lbis}
\State $h := r - l$ \label{line:hbis}
\State $I^h := \iaboundfun{h}{\K}$ \label{line:iaboundbis}
\State $\underline{l}'_k = \max \{t : l' - t \in \mathcal{H}_k(\K) \}$ \label{line:kslbelow} 
\State $\overline{l}'_k = \min \{t : t - l' \in \mathcal{H}_k(\K) \}$ \label{line:kslabove}
\State $I_k^l \gets [\varepsilon \underline{l}'_k, \varepsilon \overline{l}'_k]$ \label{line:fpksbound}
\State \Return $I_k := I_k^l + I^h$ 
\end{algorithmic}
}
\caption{\new{$\fpkristen$: our algorithm to compute roundoff errors bounds of programs implementing polynomial expressions with sparse Krivine-Stengle representations.}}
\label{alg:fpkristen}
\end{figure}

\begin{proposition}[Convergence of Algorithm~$\fpkristen$] 
\label{sparseLPrelax_th}
The sequence of optimal values $(\underline{l}'_k)$ (resp.~$(\overline{l}'_k)$) satisfies $\underline{l}'_k \uparrow \underline{l}'$ (resp.~$\overline{l}'_k \downarrow \overline{l}'$) as $k \rightarrow +\infty$. 
\new{
Thus, after running Algorithm~$\fpkristen$ on the polynomial $f$ of degree $d$,  the sequence of interval enclosures $(I_k^l)_{k \geq d+1}$ returned at Line~\lineref{line:fpksbound} converges to $I^l := [\underline{l}, \overline{l}]$.
}
\end{proposition}
\begin{proof}
By construction $(\underline{l}'_k)$ is monotone nondecreasing. 
For a given arbitrary $\varepsilon' > 0$, the polynomial $l'- \underline{l}' + \varepsilon'$ is positive over $\cSet$. By Lemma~\ref{pattern}, the subsets $I_j$ and $J_j$ satisfy the four conditions stated in~Definition~\ref{sparse_def}, so we can apply Theorem \ref{spkrivrep_th} to $l'- \underline{l}' + \varepsilon'$. This yields the existence of $\phi_j$, $j=1,\dots,m$, such that $l'- \underline{l}' + \varepsilon' = \sum_{j=1}^m\phi_j$ and $\phi_j = \sum_{|\ab_j+\bb_j|\leq k} \lambda_{\ab_j,\bb_j} h_{\ab_j,\bb_j}$, $j=1,\dots,m$.
Hence, $(\underline{l}' - \varepsilon',\phi_j,\lambda_{\ab_j,\bb_j})$ is feasible for LP~\eqref{sparsetheq}. 
It follows that there exists $k$ such that $\underline{l}'_k \geq l'-\varepsilon'$.
Since $\underline{l}'_k \leq l'$, and $\varepsilon'$ has been arbitrary chosen, we obtain the convergence result for the sequence $(\underline{l}'_k)$. The proof is analogous for $(\overline{l}'_k)$ and yields convergence to the enclosure of $l$.
\end{proof}
%

\begin{remark}
\label{rk:cost2}
\new{In the special case of roundoff error computation, one can prove that the number of variables of LP~\eqref{sparsetheq} is $m \binom {2(p+1)+k} {k} +1$ with a number of constraints equal to $[\frac{m\kh}{n+1}+1]\binom{n+\kh}{\kh}$. This is in contrast with the dense case where the number of LP variables is $\binom {2(p+m)+k} {k} +1$ with a number of  constraints equal to $\binom {n+m+\kh}{\kh}$ . }
\end{remark}

\renewcommand*{\proofname}{Proof of Remark 2}
\begin{proof}
\new{We replace the representation of a function $\phi$ of dimension $(n+m)$ on the set $\K$ by a sum of $m$ functions $\phi_j$ of dimension $(n+1)$ defined on their associated subsets $\K_j$. 
From Section~\ref{sec:globalKS}, the number of coefficients $\lambda_{\ab_j,\bb_j}$ for the K.S. representation of a $\phi_j$ over $\K_j$ is $\binom {2(p+1)+k} {k}$. 
This leads to a total of $m \binom {2(p+1)+k} {k}$ for all the $\phi_j$ and $m \binom {2(p+1)+k} {k} +1$ variables when adding $t$.}

\new{The number of equality constraints is the number of monomials involved in  $\sum_{j=1}^m\phi_j$. Each $\phi_j$ has $\binom {(n+1)+\kh} {\kh}$ monomials. However there are redundant monomials between all the $\phi_j$: the ones depending of only $\xb$, and not $\eb$. These $\binom {n+\kh} {\kh}$ monomials should appear only once. This leads to a final number of $m \binom {(n+1)+\kh} {\kh} - (m-1)\binom {n+\kh} {\kh}$ monomials which is equal to $[\frac{m\kh}{n+1}+1]\binom{n+\kh}{\kh}$. }
\end{proof}

\renewcommand*{\proofname}{Proof}
%
\begin{example}
Continuing Example~\ref{bernex}, for the polynomial $l$ defined in~\eqref{eq:l} (Section~\ref{overview}), \new{we consider LP~\eqref{sparsetheq} at the relaxation order $k = d = 3$. This problem involves $3\binom{2\times(1+1)+3}{3} +1 = 106$ variables and $[\frac{3 \times 3}{2}+1]\binom{4}{3} = 22$ constraints.} This is in contrast with a dense Krivine-Stengle representation, where the corresponding LP involves $35$ linear equalities and $166$ variables. Computing the values of $\underline{l}'_k$ and $\overline{l}'_k$ provides  an upper bound of $2$ for $|l'|$ on $\K$, yielding $|l(x,\e)| \leq 2\varepsilon$, for all $(x,\e) \in [0,1]\times [-\varepsilon,\varepsilon]^3$.
\end{example}
\new{\subsection{Convergence rate of \texttt{FPBern} and \texttt{FPKriSten}}
\label{complexity}
We investigate the convergence rate of Algorithm~$\fpbern$ presented in Section~\ref{contributions_handelman} as well as Algorithm~$\fpkristen$ presented in Section~\ref{contributions_handelman}. For this, we rely mainly on the results from~\cite{deKlerk10Error}.
We restrict our complexity analysis to the case of a program implementing a polynomial expression $f \in \R[\x]$ of multi-degree $d$ with rounded expression $\hat{f}$ and the input set of constraints is the hypercube $\X = [0, 1]^n$. We note $\K = \X \times \E$, with $\E = [-1,1]^m$. 
As shown above, for all $k \geq d+1$ and $\kb = (k,\dots,k)$, the scaled linear part $l'$ of the roundoff error $r = \hat{f}-f$ can be approximated from below either by $\underline{l}_{\kb}'$ (the minimum over the Bernstein coefficients of $l'$) or by $\underline{l}_k'$, the optimal value of LP~\eqref{sparsetheq}.

As in~\cite{deKlerk10Error}, for a polynomial $s \in \R[\x]$, with $s(\x) = \sum_{\ab} s_{\ab} \x^{\ab}$, we set $L(s) := \max_{\ab} |s_{\ab}| \frac{\alpha_1!\cdots\alpha_n!}{|\ab|!}$. 

We first recall the error bounds, given  by Theorem~3.4~(iii) and Theorem~1.4~(iii) in~\cite{deKlerk10Error}:
\begin{theorem}
\label{th:complexity}
Let $s \in \R[\x]$ of degree $d$ and $\underline{s} := \min_{\x \in \X} s(\x)$.  Then the following holds for all $k \geq d$ and $\kb=(k,\dots,k)$:
\begin{align}
\label{eq:complexity}
\underline{s} - \underline{s}_{k n}\leq \underline{s} - \underline{s}_{\kb} \leq \frac{L(s)}{k} \binom{d+1}{3} n^d \,,
\end{align}
\if{ 
\begin{align}
\label{eq:complKS}
\underline{s} - \underline{s_{k n}} \leq \frac{L(s)}{k} \binom{d+1}{3} n^d \,.
\end{align}
}\fi
\end{theorem}
We now derive similar bounds for the function $l'$ on $\K$, obtained with either Bernstein expansions or sparse Krivine-Stengle representations. By contrast with~\eqref{eq:complexity}, we obtain upper bounds for the differences $\underline{l}' - \underline{l}_\kb'$ and $\underline{l}' - \underline{l}_{k n + 1}'$ which are proportional to $1/k$, yielding linear convergence rates  w.r.t.~degree elevation for both $\fpbern$ and $\fpkristen$.
\begin{theorem}
\label{th:lcomplexity}
Let $\underline{l}' = \min_{(\x,\e) \in \K} \sum_{j=1}^m s_j(\x) e_j$ and $L_m := \sum_{1 \leq j \leq m} L(s_j)$. Then the following holds for all $k \geq d$ and $\kb=(k,\dots,k)$:
\begin{align}
\label{eq:complexityl}
\underline{l}' - \underline{l}'_{kn+1} \leq \underline{l}'- \underline{l}'_{\kb} \leq \frac{3 L_m}{k} \binom{d+1}{3} n^d \,.
\end{align}
\if{
\begin{align}
\label{eq:complexityB}
 \underline{l}'- \underline{l}'_{\kb} \leq \frac{3 L_m}{k} \binom{d+1}{3} n^d \,,
\end{align}
\begin{align}
\label{eq:complexityKS}
\underline{l}' - \underline{l}'_{kn+1} \leq \frac{3 L_m}{k} \binom{d+1}{3} n^d \,.
\end{align}
}\fi
\end{theorem}
\begin{proof}
Let us start with the right inequality from~\eqref{eq:complexityl}. We make change-of-variables by noting that 
\begin{align*}
\underline{l}' & = \min_{(\x,\e) \in \K} \sum_{j=1}^m s_j(\x) e_j \\
& = \min_{(\x,\e) \in [0, 1]^{n+m}} \sum_{j=1}^m s_j(\x) (2 e_j+1) \,.
\end{align*}
Let $s_0(\x) := \sum_{j=1}^m - s_j(\x)$. For all $j=0,\dots,m$, we note $C(k,s_j) := \frac{L(s_j)}{k} \binom{d+1}{3} n^d$.
As in~\cite[(3.5)]{deKlerk10Error}, we obtain the following decomposition for all $j=0,\dots,m$:
\begin{align}
\label{eq:decompsj}
s_j = \sum_{\ab \leq \kb} s_j(\frac{\ab}{k}) \mathbf{B}_{\kb,\ab} + h_j - C(k,s_j) \,,
\end{align}
where each polynomial $h_j$ belongs to $\mathcal{H}_{k n}(\X)$ and has nonnegative Bernstein coefficients in the basis $(\mathbf{B}_{\kb,\ab})_{\ab \leq \kb}$. Therefore, the Bernstein coefficients of the polynomials involved in~\eqref{eq:decompsj} satisfy the following, for all $\ab \leq \kb$ and for each $j=0,\dots,m$: 
\[
b_{\ab}^{(s_j)} = s_j(\frac{\ab}{k}) +  b_{\ab}^{(h_j)} - C(k,s_j) \geq s_j(\frac{\ab}{k}) - C(k,s_j) \,.
\]
So, for all $j=1,\dots,m$, for all $e_j \in [0, 1]$ and for all $\ab \leq \kb$, one has $b_{\ab}^{(s_j)} e_j \geq s_j(\frac{\ab}{k}) e_j - C(k,s_j)$.
For all $\ab \leq \kb$ and for all $\e \in [0, 1]^m$, we obtain:
\begin{align*}
b_{\ab}^{(l_\eb')} & = \sum_{j=1}^m 2 b_{\ab}^{(s_j)} e_j + b_{\ab}^{(s_0)} \\
& \geq  \sum_{j=1}^m 2 [s_j(\frac{\ab}{k}) e_j - C(k,s_j)] + s_0(\frac{\ab}{k}) - C(k,s_0) \\
& = l'(\frac{\ab}{k}) - 2 \sum_{j=1}^m C(k,s_j) -C(k,s_0) \\
& \geq \underline{l}' - \frac{3 L_m}{k} \binom{d+1}{3} n^d  \,.
\end{align*}
Since $\underline{l}_{\kb}' = \min_{\e \in [0, 1]^n} \min_{\ab \leq \kb} b_{\ab}^{(l_\eb')} $, this implies the right inequality from~\eqref{eq:complexityl}. 
%

In the remainder of the proof, we emphasize that the variable $\e$ lies in $[-1,1]^m$, so that $(\x,\e) \in \K$. To prove the left inequality from~\eqref{eq:complexityl}, we show that  $l'- \underline{l}'_{\kb}$ has a sparse Krivine-Stengle representation in $\mathcal{H}_{kn+1}(\K)$. Note that $\underline{l}' = \sum_{\ab \leq \kb} \underline{l}' \, \mathbf{B}_{\kb,\ab}$, thus we obtain the following decomposition in the Bernstein basis, for all $(\x,\e) \in \K$:
\begin{align}
\label{eq:secondKS}
l'(\x,\e) - \underline{l}'_{\kb} =  \sum_{\ab \leq \kb} [ \sum_{j=1}^m  b_{\ab}^{(s_j)} \, e_j - \underline{l}'] \, \mathbf{B}_{\kb,\ab}(\x) \,.
\end{align}
We now prove the existence of nonnegative scalars $(u_{\ab,j})_{\ab \leq \kb}$, $(v_{\ab,j})_{\ab \leq \kb}$ and $w_{\ab}$ such that
\[
\sum_{j=1}^m b_{\ab}^{(s_j)}  \, e_j - \underline{l}' = \sum_{j=1}^m u_{\ab,j} \frac{1-e_j}{2} +  \sum_{j=1}^m v_{\ab,j} \frac{1+e_j}{2} + w_{\ab} \,,
\]
which together with~\eqref{eq:secondKS} implies that $l' - \underline{l}'_{\kb} \in \mathcal{H}_{kn+1}(\K)$. 
For this, let us choose $u_{\ab,j} := |b_{\ab}^{(s_j)}| - b_{\ab}^{(s_j)}$, $v_{\ab,j} := |b_{\ab}^{(s_j)}| + b_{\ab}^{(s_j)}$ and $w_{\ab} := - \sum_{j=1}^m |b_{\ab}^{(s_j)}| - \underline{l}'_{\kb}$, so that the above equality holds. Since $|b_{\ab}^{(s_j)}| \geq b_{\ab}^{(s_j)}$ and $|b_{\ab}^{(s_j)}| \geq - b_{\ab}^{(s_j)}$, one has $u_{\ab,j} \geq 0$ and $v_{\ab,j} \geq 0$, respectively. Eventually, $- \sum_{j=1}^m |b_{\ab}^{(s_j)}| + \max_{\ab \leq \kb} \sum_{j=1}^m |b_{\ab}^{(s_j)}| \geq 0 $, which shows that $w_{\ab} \geq 0$. \\
Hence, we proved that $l' - \underline{l}'_{\kb}$ has a sparse Krivine-Stengle representation in $\mathcal{H}_{kn+1}(\K)$. Since $\underline{l}'_{kn+1} = \max \{t : l' - t \in \mathcal{H}_{kn+1}(\K) \}$, we obtain $\underline{l}'_{kn+1} \geq \underline{l}'_{\kb}$, the desired result. 

\if{
Let $C := 2 \sum_{j=1}^m C(k,s_j) + C(k,s_0) = \frac{3 L_m}{k} \binom{d+1}{3} n^d$. To prove the second inequality, we show that $l' - \underline{l}' + C$ has a sparse Krivine-Stengle representation in $\mathcal{H}_{kn+1}(\K)$. For all $s \in \R[\x]$, the Bernstein approximation of $s$ is given by $B_{\kb}(s) := \sum_{\ab \leq \kb} s(\frac{\ab}{k}) \mathbf{B}_{\kb,\ab}$. Our strategy consists of writing: 
\[l' - \underline{l}' + C = [l' - \sum_{j=1}^m B_\kb(s_j) e_j + C] +  [\sum_{j=1}^m B_\kb(s_j) e_j - \underline{l}'] \,,
\]
then to show that this is the sum of two polynomials in $\mathcal{H}_{kn+1}(\K)$.\\
For the first polynomial, we combine~\eqref{eq:decompsj}  with a reverse change-of-variables to obtain 
\[l'(\x,\e) = \sum_{j=1}^m B_\kb(s_j) e_j + \sum_{j=1}^m 2 q_j(\x) \frac{1 + e_j}{2} + q_0(\x) + C \,,
\] 
for all $(\x,\e) \in \K$. Since $q_j \in \mathcal{H}_{k n}(\X)$, for all $j=0,\dots,m$, this implies that $l' - \sum_{j=1}^m B_\kb(s_j) e_j + C \in \mathcal{H}_{k n+1}(\K)$.\\
Next, we focus on the second polynomial $\sum_{j=1}^m B_\kb(s_j) e_j - \underline{l}'$. Note that  $B_\kb(s_j) e_j = \sum_{\ab \leq \kb} s_j(\frac{\ab}{k}) \, e_j \mathbf{B}_{\kb,\ab} $ and $\underline{l}' = \sum_{\ab \leq \kb} \underline{l}' \, \mathbf{B}_{\kb,\ab}$, yielding
\begin{align}
\label{eq:secondKS}
\sum_{j=1}^m B_\kb(s_j) e_j - \underline{l}' = \sum_{\ab \leq \kb} [ \sum_{j=1}^m s_j(\frac{\ab}{k}) \, e_j - \underline{l}'] \, \mathbf{B}_{\kb,\ab} \,.
\end{align}
We now prove the existence of nonnegative scalars $(u_{\ab,j})_{\ab \leq \kb}$, $(v_{\ab,j})_{\ab \leq \kb}$ and $w_{\ab}$ such that
\[
\sum_{j=1}^m s_j(\frac{\ab}{k}) \, e_j - \underline{l}' = \sum_{j=1}^m u_{\ab,j} \frac{1-e_j}{2} +  \sum_{j=1}^m v_{\ab,j} \frac{1+e_j}{2} + w_{\ab} \,,
\]
which together with~\eqref{eq:secondKS} implies that $\sum_{j=1}^m B_\kb(s_j) e_j - \underline{l}' \in \mathcal{H}_{kn+1}(\K)$. 
For this, let us choose $u_{\ab,j} := |s_j(\frac{\ab}{k})| - s_j(\frac{\ab}{k})$, $v_{\ab,j} := |s_j(\frac{\ab}{k})| + s_j(\frac{\ab}{k})$ and $w_{\ab} := - \sum_{j=1}^m |s_j(\frac{\ab}{k})| - \underline{l}'$, so that the above equality holds. Since $|s_j(\frac{\ab}{k})| \geq s_j(\frac{\ab}{k}), - s_j(\frac{\ab}{k})$, one has $u_{\ab,j}, v_{\ab,j} \geq 0$. Eventually, $- \sum_{j=1}^m |s_j(\frac{\ab}{k})|$ is the value of $l'$ obtained by taking $e_j = -\sgn{(s_j(\frac{\ab}{k}))}$ for all $j=1,\dots,m$, which shows that $w_{\ab} \geq 0$. \\
Hence, we proved that $l' - \underline{l}' + C $ has a sparse Krivine-Stengle representation in $\mathcal{H}_{kn+1}(\K)$. Since $\underline{l}'_{kn+1} = \max \{t : l' - t \in \mathcal{H}_{kn+1}(\K) \}$, we obtain $\underline{l}'_{kn+1} \geq \underline{l}' - C$, the desired result. 
}\fi
\end{proof}
\begin{remark}
Theorem~\ref{th:lcomplexity} provides convergence rates when the degree of approximation $k$ goes to infinity. Note that $L_m$ is linear in the number of roundoff error variables $m$. Hence, the value of $k$ required to get a $\delta$-approximation of $\underline{l}$ is linear in $\frac{1}{\delta}$ and $m$, polynomial in the number of input variables $n$ and exponential in the degree $d$. \\
By using Remark~\ref{rk:cost1}, the number of Bernstein coefficients of multi-degree less than  $\kb = (k,\dots,k)$ mandatory to compute this $\delta$-approximation is linear in $m$ and proportional to $(k+1)^{n}$. Similarly, by using Remark~\ref{rk:cost2}, the size of the LP~\eqref{sparsetheq} is linear in $m$ and proportional to $\binom{2(n+1)+k}{k}$. Thanks to the following lower bound:
\begin{align*}
\binom{2(n+1)+k}{k} &= \frac{[2(n+1)+k] \cdots [k+1]}{[2(n+1)]!}  \\
& = (1+\frac{k}{2n+1})(1+\frac{k}{2n})\cdots(1+k) \\ & \leq k^{2n} (1+k) \leq 2 k^{2n+1} \,,
\end{align*} 
for all $k \geq 2$, we conclude that the size of the LP relaxations to compute Krivine-Stengle representations has the same order of magnitude that the number of Bernstein coefficients. Modern LP solvers rely on interior-point methods with polynomial-time 
complexity in the LP size (see e.g.~\cite{NN94}). 
The overall theoretical arithmetic cost of both algorithms is polynomial in $m$ and exponential to $n$ and $d$. 

However in practice, the degree $k$ is fixed for the sake of efficiency. In this case, one can write   $\binom{2(n+1)+k}{k} \leq 2 (2n+1)^k$ and the size of LP relaxations is polynomial in $n$.
Therefore, the computational cost at fixed $k$ is exponential in $n$ for Bernstein expansions and polynomial in $n$ for Krivine-Stengle representations.
\end{remark}}
\section{Implementation \& Results}
\label{implementation}
\subsection{The  \texttt{FPBern} and \texttt{FPKriSten} software packages}
 We provide two distinct software packages to compute certified error bounds of roundoff errors  for programs implementing  polynomial functions with floating point precision. 
The first tool \href{https://github.com/roccaa/FPBern}{\texttt{FPBern}} relies on the method from Section \ref{contributions_bernstein} and the second tool \href{https://github.com/roccaa/FPKriSten}{\texttt{FPKriSten}} on the method from Section \ref{contributions_handelman}.\newline
\texttt{FPBern} is built on top of the \new{\texttt{C++}} software presented in \cite{dreossiHSCC} to manipulate Bernstein expansions. \new{\texttt{FPBern}} includes two modules: \texttt{FPBern(a)} and \texttt{FPBern(b)}. 
Their main difference is that Bernstein coefficients are computed with double precision floating point arithmetic in \texttt{FPBern(a)} and with rational arithmetic in \texttt{FPBern(b)}. 
Polynomial operations \new{and rational arithmetic operations }are handled with {\sc GiNaC}~\cite{ginac}. 
%
$\fpkristen$ is built on top of the {\sc SBSOS} software related to~\cite{sbsos} which handles sparse polynomial optimization problems by solving a hierarchy of convex relaxations. \new{This hierarchy is obtained by mixing Krivine-Stengle  and Putinar representations of positive polynomials.  
To improve the overall performance in our particular case, we only consider the former representation yielding the hierarchy of LP relaxations~\eqref{sparsetheq}.} Among several LP solvers,  {\sc Cplex}~\cite{cplex} yields the best performance in our case (see also~\cite{lp_compare} for more comparisons). Polynomials are handled with the {\sc Yalmip} toolbox~\cite{YALMIP} available within \texttt{Matlab}. Even though the semantics of programs considered in this paper is actually much simpler than that considered by other tools such as \texttt{Rosa}~\cite{rosa} or {\sc Fluctuat}~\cite{fluctuat}, we emphasize that those tools may be combined with external non-linear solvers to solve specific sub-problems, a task that either \texttt{FPBern} or \texttt{FPKriSten} can fulfill.
%
\subsection{Experimental results}
We tested our two software packages with \new{$35$} programs \new{(see Appendix~\ref{appendix})} where \new{$27$} are existing benchmarks coming from biology, space control and optimization fields, and $8$ are generated as follows, with $\xb = (x_1,\dots,x_n) \in [-1,1]^n$.
\begin{equation}
	\texttt{ex-n-nSum-deg}(\xb) := \sum_{j=0}^{\texttt{nSum}}(\prod_{k=1}^{\texttt{deg}}(\sum_{i=1}^{\texttt{n}}x_i)) \,.
\end{equation}
The first $9$ \new{and the last $15$ programs }are used  \new{for similar comparison} in~\cite[Section 4.1]{real2float}. \new{Additionally, $3$ benchmarks }come from~\cite{SolovyevH13}. The 8 generated benchmarks allow evaluating \emph{independently} the performance of the tools w.r.t.~either the number of input variables (through the variable \texttt{n}), the degree (through \texttt{deg}) or the number of error variables (through \texttt{nSum}). 
Taking $\xb \in [-1,1]^n$ allows avoiding monotonicity of the polynomial \new{(which could be exploited by the Bernstein techniques)}.\\
We recall that each program implements a polynomial function $f(\xb)$ with box constrained  input variables. To provide an upper bound of the absolute roundoff error $| f(\xb)-\hat{f}(\xb,\eb) | = | l(\xb,\eb) + h(\xb,\eb) |$, we rely on~\texttt{Real2Float} to generate  $l$ and to bound $h$ (see~\cite[Section 3.1]{real2float}). Then the optimization methods of Section~\ref{contributions} are applied to bound a function $l'$, obtained after linear transformation of $l$, over the unit box. \newline
At a given multi-degree $\kb$, Algorithm~$\fpbern$ computes the  bound $\overline{l}'_{\kb}$ (\new{see Figure~\ref{alg:fpbern}}). Similarly, at a given relaxation order $k$, Algorithm~$\fpkristen$ computes the bounds $\underline{l}'_k$ and $\underline{l}'_k$ (\new{see Figure~\ref{alg:fpkristen}}). 
To achieve fast computations,  the default value of $\kb$ is the multi-degree $\db$ of $l'_{\eb}$ (equal to the multi-degree of the input polynomial $f$) and the default value of $k$ is the degree $d+1$ of $l'$ (equal to the successor of the degree of $f$).\newline
\new{All the experiments, with the exception to \texttt{floudas2-6}, }were carried out on an Intel Core i7-5600U (2.60Ghz, 16GB) with Ubuntu 14.04LTS, \texttt{Matlab 2015a}, {\sc GiNaC}~1.7.1, and {\sc Cplex}~12.63. \new{The execution of \texttt{floudas2-6} was performed on a different setting as it required 28GB of memory. For this reason its associated performance appears in \textit{italic} in Table~\ref{table:cpu}. }
Our benchmark settings are similar to~\cite[Section 4]{real2float} as we compare the accuracy and execution times of our two tools with \texttt{Rosa real compiler}~\cite{rosa} (version from May 2014), \texttt{Real2Float}~\cite{real2float} (version from July 2016) and \texttt{FPTaylor}~\cite{fptaylor} (version from May 2016) on programs implemented in double precision while considering input variables as real variables. All these tools use a simple rounding model (see Section~\ref{preliminaries_floating_point}) and were executed in this experiment with their default parameters.\newline
\begin{table*}[!t]
\begin{center}
\small
\caption{Comparison results of upper bounds for absolute roundoff errors. 
The best results are emphasized using \textbf{bold fonts}.\label{table:error}}{

\new{
\begin{tabular}{lccc|ccc|ccc}
\hline
{Benchmark} &  $n$ & $m$ &  $d$ & {\texttt{FPBern}(a)}  & {\texttt{FPBern}(b)}  & {\texttt{FPKriSten}}  & {\texttt{Real2Float}} & {\texttt{Rosa}}  & {\texttt{FPTaylor}} \\
\hline  
\multicolumn{10}{c}{{Programs implementing polynomial functions with variables in boxes}} \\                
\hline
{\texttt{rigidBody1}}& 3 & 10 & 3
& $5.33\text{e--}13$ & $5.33\text{e--}13$ & $5.33\text{e--}13$ & $5.33\text{e--}13$ & $5.08\text{e--}13$ & $\mathbf{3.87\textbf{e--}13}$   \\
{\texttt{rigidBody2}}& 3 & 15 & 5
& $6.48\text{e--}11$ & $6.48\text{e--}11$ & $6.48\text{e--}11$ & $6.48\text{e--}11$ & $6.48\text{e--}11$ & $\mathbf{5.24\textbf{e--}11}$  \\
{\texttt{kepler0}}& 6 & 21 & 3 
& ${1.08\text{e--}13}$ & ${1.08\text{e--}13}$ & ${1.08\text{e--}13}$ & ${1.18\text{e--}13}$ & $1.16\text{e--}13$ & $\mathbf{1.05\textbf{e--}13}$   \\
{\texttt{kepler1}}& 4 & 28 & 4  
& $\mathbf{4.23\textbf{e--}13}$ & $\mathbf{4.23\textbf{e--}13}$ & $\mathbf{4.23\textbf{e--}13}$ & $4.47\text{e--}13$ & $6.49\text{e--}13$ & ${4.49\text{e--}13}$    \\
{\texttt{kepler2}}& 6 & 42 & 4
& $\mathbf{2.03\textbf{e--}12}$ & $\mathbf{2.03\textbf{e--}12}$ & $\mathbf{2.03\textbf{e--}12}$ & $2.09\text{e--}12$ & $2.89\text{e--}12$ & ${2.10\text{e--}12}$   \\
{\texttt{sineTaylor}} &  1 & 13 & 8
& $\mathbf{5.51\textbf{e--}16}$ & $\mathbf{5.51\textbf{e--}16}$ & $\mathbf{5.51\textbf{e--}16}$ & $6.03\text{e--}16$ & $9.56\text{e--}16$ & $6.75\text{e--}16$  \\
{\texttt{sineOrder3}}& 1 & 6 & 4
& $1.35\text{e--}15$ & $1.35\text{e--}15$ & $1.25\text{e--}15$ & $1.19\text{e--}15$ & $1.11\text{e--}15$ & $\mathbf{9.97\textbf{e--}16}$   \\
{\texttt{sqroot}} & 1 & 15 & 5
& $1.29\text{e--}15$ & $1.29\text{e--}15$ & $1.29\text{e--}15$ & $1.29\text{e--}15$ & $8.41\text{e--}16$ & $\mathbf{7.13\textbf{e--}16}$  \\
{\texttt{himmilbeau}}& 2 & 11 & 5
& ${2.00\text{e--}12}$ & ${2.00\text{e--}12}$ & ${1.97\text{e--}12}$ & ${1.43\text{e--}12}$ & ${1.43\text{e--}12}$ & $\mathbf{1.32\textbf{e--}12}$    \\
{\texttt{schwefel}}& 3 & 15 & 5 
& ${1.48\text{e--}11}$ & ${1.48\text{e--}11}$ & ${1.48\text{e--}11}$ & ${1.49\text{e--}11}$ & ${1.49\text{e--}11}$ & $\mathbf{1.03\textbf{e--}11}$    \\
{\texttt{magnetism}} & 7 & 27 & 3 
& ${1.27\text{e--}14}$ & ${1.27\text{e--}14}$ & ${1.27\text{e--}14}$ & ${1.27\text{e--}14}$ & ${1.27\text{e--}14}$ & $\mathbf{7.61\textbf{e--}15}$    \\
{\texttt{caprasse}}& 4 & 34 & 5
& ${4.49\text{e--}15}$ & ${4.49\text{e--}15}$ & ${4.49\text{e--}15}$ & ${5.63\text{e--}15}$ & ${5.96\text{e--}15}$ & $\mathbf{3.04\textbf{e--}15}$    \\
\hline
\hline
{\texttt{ex-2-2-5}}& 2 & 9 & 3
& ${2.23\text{e--}14}$ & ${2.23\text{e--}14}$ & ${2.23\text{e--}14}$  & ${2.23\text{e--}14}$ & ${2.23\text{e--}14}$ & $\mathbf{1.96\textbf{e--}14}$   \\
{\texttt{ex-2-2-10}}& 2 & 14 & 3
& ${5.33\text{e--}14}$ & ${5.33\text{e--}14}$ & ${5.33\text{e--}14}$ & ${5.33\text{e--}15}$ & ${5.33\text{e--}14}$ & $\mathbf{4.85\textbf{e--}14}$  \\
{\texttt{ex-2-2-15}}& 2 & 19 & 3
& ${9.55\text{e--}14}$ & ${9.55\text{e--}14}$ & ${9.55\text{e--}14}$ & ${9.55\text{e--}14}$ & ${9.55\text{e--}14}$ & $\mathbf{8.84\textbf{e--}14}$   \\
{\texttt{ex-2-2-20}}& 2 & 24 & 3
& ${1.49\text{e--}13}$ & ${1.49\text{e--}13}$ & ${1.49\text{e--}13}$ & $\texttt{TIMEOUT}$ &  ${1.49\text{e--}13}$ & $\mathbf{1.40\textbf{e--}13}$    \\
{\texttt{ex-2-5-2}}& 2 & 9 & 6
& ${1.67\text{e--}13}$ & ${1.67\text{e--}13}$ & ${1.67\text{e--}13}$ & ${1.67\text{e--}13}$ & ${1.67\text{e--}13}$ & $\mathbf{1.41\textbf{e--}13}$    \\
{\texttt{ex-2-10-2}}& 2 & 14 & 11
& ${1.05\text{e--}11}$ & ${1.05\text{e--}11}$ & ${1.34\text{e--}11}$ & ${1.05\text{e--}11}$ & ${1.05\text{e--}11}$ & $\mathbf{8.76\textbf{e--}12}$    \\
{\texttt{ex-5-2-2}}& 5 & 12 & 3
& ${8.55\text{e--}14}$ & ${8.55\text{e--}14}$ & ${8.55\text{e--}14}$ & ${8.55\text{e--}14}$ & ${8.55\text{e--}14}$ & $\mathbf{7.72\textbf{e--}14}$    \\
{\texttt{ex-10-2-2}}& 10 & 22 & 3
& ${5.16\text{e--}13}$ & ${5.16\text{e--}13}$ & ${5.16\text{e--}13}$ & ${5.16\text{e--}13}$ & ${5.16\text{e--}13}$ & $\mathbf{4.82\textbf{e--}13}$    \\
\hline
\multicolumn{10}{c}{{Programs implementing polynomial functions with variables in basic compact semialgebraic sets}} \\  
\hline
{\texttt{floudas2-6}}& 10 & 50 & 3 & $-$
& $-$ & $\mathbf{4.34\textbf{e--}13}$ & ${5.15\text{e--}13}$ & ${5.87\text{e--}13}$ & ${7.88\text{e--}13}$  \\
{\texttt{floudas3-3}}& 6 & 25 & 3 & $-$
& $-$ & $\mathbf{4.05\textbf{e--}13}$ & ${5.81\text{e--}13}$ & $\mathbf{4.05\textbf{e--}13}$ & ${5.76\text{e--}13}$  \\
{\texttt{floudas3-4}}& 3 & 7 & 3 & $-$
& $-$ & ${2.67\text{e--}15}$ & ${2.78\text{e--}15}$ & ${2.56\text{e--}15}$ & $\mathbf{2.23\textbf{e--}15}$  \\
{\texttt{floudas4-6}}& 2 & 4 & 3 & $-$
& $-$ & ${1.89\text{e--}15}$ & ${1.82\text{e--}15}$ & ${1.33\text{e--}15}$ & $\mathbf{1.23\textbf{e--}15}$  \\
{\texttt{floudas4-7}}& 2 & 8 & 3 & $-$
& $-$ & ${2.07\text{e--}14}$ & $\mathbf{1.06\textbf{e--}14}$ & ${1.31\text{e--}14}$ & ${1.80\text{e--}14}$ \\
\hline
\multicolumn{10}{c}{{Programs implementing rational functions with variables in boxes}} \\  
\hline
{\texttt{doppler1}}& 3 & 11 & 3 &  ${1.65\text{e--}13}$
& ${1.65\text{e--}13}$ & $-$ & ${7.65\text{e--}12}$ & ${4.92\text{e--}13}$ & $\mathbf{1.59\textbf{e--}13}$  \\
{\texttt{doppler2}}& 3 & 11 & 3 &  ${3.14\text{e--}13}$
& ${3.14\text{e--}13}$ & $-$ & ${1.57\text{e--}11}$ & ${1.29\text{e--}12}$ & $\mathbf{2.90\textbf{e--}13}$  \\
{\texttt{doppler3}}& 3 & 11 & 3 &  $\mathbf{8.14\textbf{e--}14}$
& $\mathbf{8.14\textbf{e--}14}$ & $-$ & ${8.55\text{e--}12}$ & ${2.03\text{e--}13}$ & ${8.22\text{e--}14}$  \\
{\texttt{verhulst}}& 1 & 5 & 5 &  ${4.40\text{e--}16}$
& ${4.40\text{e--}16}$ & $-$ & ${4.67\text{e--}16}$ & ${6.82\text{e--}16}$ & $\mathbf{3.53\textbf{e--}16}$  \\
{\texttt{carbonGas}}& 1 & 11 & 4 &  ${1.42\text{e--}08}$
& ${1.42\text{e--}08}$ & $-$ & ${2.21\text{e--}08}$ & ${4.64\text{e--}08}$ & $\mathbf{1.23\textbf{e--}08}$ \\ 
{\texttt{predPrey}}& 1 & 7 & 10 &  ${2.32\text{e--}16}$
& ${2.32\text{e--}16}$ & $-$ & ${2.52\text{e--}16}$ & ${2.94\text{e--}16}$ & $\mathbf{1.89\textbf{e--}16}$ \\ 
{\texttt{turbine1}}& 3 & 17 & 4 &  ${7.75\text{e--}14}$
& ${7.75\text{e--}14}$ & $-$ & ${2.45\text{e--}11}$ & ${1.25\text{e--}13}$ & $\mathbf{2.33\textbf{e--}14}$ \\ 
{\texttt{turbine2}}& 3 & 13 & 2 &  ${1.16\text{e--}13}$
& ${1.16\text{e--}13}$ & $-$ & ${2.08\text{e--}12}$ & ${1.76\text{e--}13}$ & $\mathbf{3.14\textbf{e--}14}$ \\ 
{\texttt{turbine3}}& 3 & 17 & 4 &  ${5.36\text{e--}14}$
& ${5.36\text{e--}14}$ & $-$ & ${1.71\text{e--}11}$ & ${8.50\text{e--}14}$ & $\mathbf{1.70\textbf{e--}14}$ \\ 
{\texttt{jet}}& 2 & 24 & 8 &  ${2.73\text{e--}09}$
& ${2.73\text{e--}09}$ & $-$ & $\text{OoM}$ & ${1.62\text{e--}08}$ & $\mathbf{1.50\textbf{e--}11}$ \\
\hline
\end{tabular}
}}
\end{center}
\end{table*}

Table \ref{table:error}  shows the result of the absolute roundoff error while Table~\ref{table:cpu} displays execution times obtained through averaging over $5$ runs. For each benchmark, we indicate the number $n$ (resp.~$m$) of input (resp.~error) variables as well as the degree $d$ of $l'$. For \texttt{FPKriSten} the {\sc Cplex} solving time in Table~\ref{table:cpu} is given between parentheses. Note that the overall efficiency of the tool could be improved by constructing the hierarchy of LP~\eqref{sparsetheq} with a \texttt{C++} implementation.

Our two methods yield more accurate bounds for \new{$3$  benchmarks implementing polynomial function with input variables in boxes:} \texttt{kepler1}, \texttt{sineTaylor} and \texttt{kepler2} which is  the program involving the largest number of error variables.\\
\new{For \texttt{kepler1}, \texttt{FPBern} and \texttt{FPKriSten} are $6\%$ more accurate than \texttt{Real2Float} and \texttt{FPTaylor}, and  $53\%$ more accurate than \texttt{Rosa}.}
For \texttt{kepler2}, our two tools are $3\%$ (resp.~$42\%$) more accurate than \texttt{FPTaylor} and \texttt{Real2Float} (resp.~\texttt{Rosa}). 
In addition, Property~\ref{sharp_prop} holds for these three programs with \texttt{FPBern}, which ensures bound optimality. 
For all other benchmarks \texttt{FPTaylor} provides the most accurate upper bounds.  
Our tools are more accurate than \texttt{Real2Float} except for \texttt{sineOrder3} and \texttt{himmilbeau}. 
In particular, for \texttt{himmilbeau}, \texttt{FPBern} and \texttt{FPKriSten} are $40\%$ (resp.~$50\%$) less accurate than  \texttt{Real2Float} (resp.~\texttt{FPTaylor}). 
One way to obtain better bounds would be to increase the degree $\kb$ (resp.~ relaxation order $k$) within \texttt{FPBern} (resp.~\texttt{FPKriSten}). 
Preliminary experiments indicate modest accuracy improvement at the expense of performance. 
We refer to Section~\ref{complexity} for theoretical results on the convergence rates of both methods.

\begin{table*}[!t]
\begin{center}
\small
\caption{Comparison of execution times (in seconds) for absolute roundoff error bounds. For \texttt{FPKriSten} the {\sc Cplex} solving time is given between parentheses.
For each model, the best results are emphasized using \textbf{bold fonts}.
\label{table:cpu}}{

\new{
\begin{tabular}{lccc|ccc|ccc}
\hline
{Benchmark} &  $n$ & $m$ & $d$ & {\texttt{FPBern}(a)}  & {\texttt{FPBern}(b)}  & {\texttt{FPKriSten}}  & {\texttt{Real2Float}} & {\texttt{Rosa}}  & {\texttt{FPTaylor}} \\
\hline
\multicolumn{10}{c}{{Programs implementing polynomial functions with variables in boxes}} \\                
\hline
{\texttt{rigidBody1}}& 3 & 10 & 3  &  $\mathbf{3\textbf{e--}4}$
& $5\text{e--}4$ & $0.24(0.03)$ & $0.58$ & ${0.13}$ & $1.84$  \\
{\texttt{rigidBody2}}& 3 & 15 & 5 & $\mathbf{1\textbf{e--}3}$
& $\mathbf{1\textbf{e--}3}$ & $2.75(0.40)$ &${0.26}$ & $2.17$ & $3.01$  \\
{\texttt{kepler0}}& 6 & 21 & 3 & $\mathbf{3\textbf{e--}3}$
& $1\text{e--}2$ & $1.64(0.11)$ & ${0.22}$ & $3.78$ & $4.93$ \\
{\texttt{kepler1}}& 4 & 28 & 4  & $\mathbf{5\textbf{e--}3}$
& $1\text{e--}2$ & ${3.88(0.48)}$ & $17.6$ & $63.1$ & $9.33$  \\
{\texttt{kepler2}}& 6 & 42 & 4  & $\mathbf{4\textbf{e--}2}$
& $0.44$ & $18.8(2.4)$ & ${16.5}$ & $106$ & $19.1$  \\
{\texttt{sineTaylor}} &  1 & 13 & 8  & $\mathbf{5\textbf{e--}4}$
& $2\text{e--}3$ & $0.86(0.19)$ & $1.05$ & $3.50$ & $2.91$  \\
{\texttt{sineOrder3}}& 1 & 6 & 4  & $\mathbf{1\textbf{e--}4}$
& $\mathbf{1\textbf{e--}4}$ & ${0.08(0.02)}$ & $0.40$ & $0.48$ & $1.90$ \\
{\texttt{sqroot}} & 1 & 15 & 5  & $\mathbf{2\textbf{e--}4}$
& $\mathbf{2\textbf{e--}4}$ & $0.26(0.05)$ & ${0.14}$ & $0.77$ & $2.70$  \\
{\texttt{himmilbeau}}& 2 & 11 & 5  & $\mathbf{1\textbf{e--}3}$
& $2\text{e--}3$ & $0.66(0.09)$ & ${0.20}$ & $2.51$ & $3.28$ \\
{\texttt{schwefel}}& 3 & 15 & 5  & $\mathbf{2\textbf{e--}3}$
& $3\text{e--}3$ & ${2.84(0.53)}$ & ${0.23}$ & $3.91$ & $0.53$ \\
{\texttt{magnetism}} & 7 & 27 & 3  & $\mathbf{6\textbf{e--}2}$
& $2.10$ & $2.99(0.18)$ & ${0.29}$ & $1.95$ & $5.91$  \\
{\texttt{caprasse}}& 4 & 34 & 5 & $\mathbf{6\textbf{e--}3}$
& $1\text{e--}2$ & $17.2(3.62)$ & ${3.63}$ & $17.6$ & $12.2$ \\
\hline
\hline
{\texttt{ex-2-2-5}}& 2 & 9 & 3 &  $\mathbf{3\textbf{e--}4}$
& $\mathbf{3\textbf{e--}4}$ & $0.13(0.02)$ & $0.07$ & $4.20$ & $2.30$  \\
{\texttt{ex-2-2-10}}& 2 & 14 & 3 &  $\mathbf{4\textbf{e--}4}$
& $\mathbf{4\textbf{e--}4}$ & $0.18(0.02)$ & $0.35$ & $4.75$ & $3.42$  \\
{\texttt{ex-2-2-15}}& 2 & 19 & 3 &  $\mathbf{5\textbf{e--}4}$
& $\mathbf{5\textbf{e--}4}$ & $0.24(0.03)$ & $9.75$ & $5.33$ & $4.91$  \\
{\texttt{ex-2-2-20}}& 2 & 24 & 3 &  $\mathbf{5\textbf{e--}4}$
& $8\text{e--}4$ & $0.30(0.03)$ & $\texttt{TIMEOUT}$ & $6.28$ & $6.27$  \\
{\texttt{ex-2-5-2}}& 2 & 9 & 6 &  $\mathbf{2\textbf{e--}3}$
& $3\text{e--}3$ & $1.08(0.14)$ & $0.27$ & $4.26$ & $2.53$  \\
{\texttt{ex-2-10-2}}& 2 & 14 & 11 &  $\mathbf{2\textbf{e--}2}$
& $4\text{e--}2$ & $90.1(53.1)$ & $49.2$ & $9.37$ & $5.07$  \\
{\texttt{ex-5-2-2}}& 5 & 12 & 3 &  $\mathbf{7\textbf{e--}3}$
& $4\text{e--}2$ & $0.63(0.05)$ & $0.21$ & $4.45$ & $12.3$  \\
{\texttt{ex-10-2-2}}& 10 & 22 & 3 &  $\mathbf{2.48}$
& $1242$ & $5.5(0.3)$ & $30.7$ & ${5.34}$ & $34.6$  \\
\hline
\multicolumn{10}{c}{{Programs implementing polynomial functions with variables in basic compact semialgebraic sets}} \\  
\hline
{\texttt{floudas2-6}}& 10 & 50 & 3 & $-$
& $-$ & $\mathit{142(25.2)}$ & $\mathbf{2.49}$ & $159$ & $15.9$  \\
{\texttt{floudas3-3}}& 6 & 25 & 3 & $-$
& $-$ & $15.2(1.24)$ & $\mathbf{0.45}$ & $13.9$ & $5.64$  \\
{\texttt{floudas3-4}}& 3 & 7 & 3 & $-$
& $-$ & $0.14(0.02)$ & $\mathbf{0.09}$ & $0.49$ & $1.47$  \\
{\texttt{floudas4-6}}& 2 & 4 & 3 & $-$
& $-$ & $0.08(0.02)$ & $\mathbf{0.07}$ & $1.20$ & $0.91$  \\
{\texttt{floudas4-7}}& 2 & 8 & 3 & $-$
& $-$ & $0.29(0.03)$ & $\mathbf{0.13}$ & $21.8$ & $1.64$ \\
\hline
\multicolumn{10}{c}{{Programs implementing rational functions with variables in boxes}} \\  
\hline
{\texttt{doppler1}}& 3 & 11 & 3 &  $\mathbf{8\textbf{e--}3}$
& $\mathbf{8\textbf{e--}3}$ & $-$ & $6.80$ & $6.35$ & $6.13$  \\
{\texttt{doppler2}}& 3 & 11 & 3 &  $\mathbf{6\textbf{e--}3}$
& $7\text{e--}3$ & $-$ & $6.96$ & $6.54$ & $6.88$  \\
{\texttt{doppler3}}& 3 & 11 & 3 &  $\mathbf{7\textbf{e--}3}$
& $\mathbf{7\textbf{e--}3}$ & $-$ & $6.84$ & $6.37$ & $9.13$  \\
{\texttt{verhulst}}& 1 & 5 & 5 &  $\mathbf{1\textbf{e--}3}$
& $2\text{e--}3$ & $-$ & $0.51$ & $1.36$ & $1.37$  \\
{\texttt{carbonGas}}& 1 & 11 & 4 &  $\mathbf{1\textbf{e--}3}$
& $\mathbf{1\textbf{e--}3}$ & $-$ & $0.83$ & $6.59$ & $3.73$ \\ 
{\texttt{predPrey}}& 1 & 7 & 10 &  $\mathbf{4\textbf{e--}3}$
& $5\text{e--}3$ & $-$ & $0.87$ & $4.12$ & $1.78$ \\ 
{\texttt{turbine1}}& 3 & 17 & 4 &  $\mathbf{4\textbf{e--}2}$
& $\mathbf{4\textbf{e--}2}$ & $-$ & $72.2$ & $3.09$ & $4.38$ \\ 
{\texttt{turbine2}}& 3 & 13 & 2 &  $\mathbf{1\textbf{e--}2}$
& $\mathbf{1\textbf{e--}2}$ & $-$ & $4.72$ & $7.75$ & $3.25$ \\ 
{\texttt{turbine3}}& 3 & 17 & 4 &  $\mathbf{4\textbf{e--}2}$
& $\mathbf{4\textbf{e--}2}$ & $-$ & $74.5$ & $4.57$ & $3.46$ \\ 
{\texttt{jet}}& 2 & 24 & 8 &  $\mathbf{3\textbf{e--}2}$
& $\mathbf{3\textbf{e--}2}$ & $-$ & $\text{OoM}$ & $125$ & $9.79$ \\
\hline
\end{tabular}
}}
\end{center}
\end{table*}
%


\texttt{FPBern(a)} is the fastest for all the benchmarks while having a similar accuracy to \texttt{Real2Float} or \texttt{Rosa}. 
\new{\texttt{FPBern(b)} has performance close to \texttt{FPBern(a)} with the exception of high dimensional benchmarks involving numerous rational arithmetic operations: \texttt{kepler2, magnetism } and the generated benchmark \texttt{ex-10-2-2}.}\\
%
%
The results obtained with the $8$ generated benchmarks emphasize the limitations of each method. The Bernstein method performs very well when the number of input variables is low, even if the degree increases, as shown in the results for the 6 programs from \texttt{ex-2-2-5} to \texttt{ex-2-10-2}. This is related to the polynomial dependency on the degree when fixing the number of input variables.
However, for the last $2$ programs \texttt{ex-5-2-2} and \texttt{ex-10-2-2} where the dimension increases, the computation time increases exponentially, and this especially visible on the \texttt{FPBern(b)}\new{ rational arithmetic implementation.} This confirms the theoretical result stated in~Remark~\ref{rk:cost1} as the number of Bernstein coefficients is exponential w.r.t.~the dimension at fixed degree.

On the same programs, the method based on Krivine-Stengle representations performs better when the dimension increases, at fixed degree. This confirms the constraint dependency on $[\frac{m k}{n+1}+1] \binom{n+k}{k}$ stated in Remark~\ref{rk:cost2}. 

Results for the $4$ programs from \texttt{ex-2-2-5} to \texttt{ex-2-2-20} also indicate that our methods are the least sensible to an increase of error variables.
We note that \texttt{FPKriSten} is often the second fastest tool.

\new{The $5$ benchmarks implementing polynomial functions with inputs variables in semialgebraic sets are only handled by \texttt{FPKriSten} in the current state. Our tool is the most accurate on the 2 benchmarks \texttt{floudas2-6} and \texttt{floudas3-3}.}
\new{For \texttt{floudas2-6}, \texttt{FPKriSten} is respectively $18\%$, $35\%$ and $81\%$ more accurate than  \texttt{Real2Float}, \texttt{Rosa} and \texttt{FPTaylor}. }
\new{For \texttt{floudas3-3}, our tool is as accurate as \texttt{Rosa} while being $43\%$ (resp.~$42\%$) more accurate than\texttt{Real2Float} (resp.~\texttt{FPTaylor}).}

\new{For benchmarks \texttt{floudas3-4} and \texttt{floudas4-6}, \texttt{FPTaylor} provides the best bounds while being $20\%$ and $54\%$ more accurate than \texttt{FPKriSten}. Finally, the most accurate tool on \texttt{floudas4-7} is \texttt{Real2Float}, being $95\%$ more accurate than our tool.}
\new{On these $5$ benchmarks, the tool with the best performance is \texttt{Real2Float}. However, with the exception of \texttt{floudas2-6}, \texttt{FPKriSten} performances are similar to \texttt{FPTaylor} and \texttt{Rosa}.}

\new{Finally we compare \texttt{FPBern} with \texttt{Real2Float}, \texttt{Rosa}, and \texttt{FPTaylor} on the $10$ benchmarks implementing rational functions. Both \texttt{FPBern(a)} and \texttt{FPBern(b)} demonstrate the best performance by a large margin on these benchmarks. We note that the performance of \texttt{FPBern(b)} (implemented in rational arithmetic) are similar to \texttt{FPBern(a)} (implemented in double precision). This can be explained by the fact that all related programs have a small number of input variables.}

\new{Our tool is the most precise on benchmark \texttt{doppler3}, being respectively $105$ and $2.5$ times more accurate than \texttt{Real2Float} and \texttt{Rosa}. 
Our bounds for \texttt{doppler3} are a slightly tighter than \texttt{FPTaylor} results, being $1\%$ more accurate. 
On the remaining $9$ benchmarks \texttt{FPBern} accuracy is the second best after \texttt{FPTaylor}. \texttt{FPBern} computes bounds closest to \texttt{FPTaylor} results on \texttt{doppler1} where \texttt{FPTaylor} is $3\%$ more accurate. Its worst accuracy, with regard to \texttt{FPTaylor} bounds, is on benchmark \texttt{jet} being $182$ times less accurate than \texttt{FPTaylor}. Overall, \texttt{FPBern} shows a very good trade-off between accuracy and performance on all $10$ benchmarks.}

Let us now provide an overall evaluation of our tools. Our tools are comparable with \texttt{Real2Float} (resp.~\texttt{Rosa}) in terms of accuracy and faster than them. In comparison with \texttt{FPTaylor}, our tools are in general less precise but still very competitive in accuracy, and they outperform \texttt{FPTaylor} in computation time. A salient advantage of our tools, in particular \texttt{FPKriSten}, over \texttt{FPTaylor} is a good trade-off between computation time and accuracy for large polynomials \new{and convex semialgebraic sets}. As we can see from the experimental results, for \texttt{ex-10-2-2}, \texttt{FPKriSten} takes only $6.11$s while \texttt{FPTaylor} takes $34.6$s for comparable precisions. Note that the experimentations were done \new{with \texttt{FPKriSten} implemented in (interpreted) \texttt{Matlab}}; \new{a }\texttt{C++} implementation of this method would allow a significant speed-up by decreasing the problem construction time, thus tightening the gap between solving time and overall time. \new{We also note that \texttt{FPBern(a)} and \texttt{FPBern(b)} achieve the same bounds for \emph{all} benchmarks.} 

\new{We emphasize that }the good time performances of our tools come from the exploitation of sparsity. \new{Indeed, a direct Bernstein expansion of the polynomial $l$ associated to \texttt{kepler2} leads to compute  $3^6 \times 2^{42}$ coefficients against $42 \times 3^6$ with \texttt{FPBern}. Similarly, dense Krivine-Stengle representations yield an LP with $\binom{100}{4} + 1 = 3 \ 921 \ 226$ variables while LP~\eqref{sparsetheq} involves $42 \binom{18}{4} + 1 = 128 \ 521$ variables.}

\section{Conclusion and Future Works}
\label{conclusion}
We propose two new methods to compute upper bounds of absolute roundoff errors occurring while
executing programs \new{involving polynomial or rational functions} with floating point precision. The first method uses Bernstein expansions of polynomials while the second one relies on a hierarchy of LP relaxations derived from sparse Krivine-Stengle representations. The overall computational cost is drastically reduced compared to the dense problem, thanks to a specific exploitation of the sparsity pattern between input and error variables, \new{yielding promising experimental results. We also provide a complexity analysis in the case of polynomial programs with box constrained variables. For both methods, this analysis allows to derive convergence rates towards the maximal value of the linear part of the roundoff error. There is a large gap between theorey and practice: the theoretical error bounds are exponential w.r.t.~the size of the programs, which is in deep contrast with the practical experiments providing tight error bounds very often}.

\new{While our second method allows to handle general polynomial programs with semialgebraic input sets, 
our first method is currently limited to programs implementing rational functions with box constrained variables. It would be worth adapting the techniques described in~\cite{Mantzaflaris2010} to obtain polygonal approximations of semialgebraic input sets.}
Next, we intend to aim at formal verification of bounds by interfacing either~\texttt{FPBern} with the PVS libraries~\cite{munoz13} related to Bernstein expansions, or \texttt{FPKirSten} with the $\texttt{Coq}$ libraries available in $\texttt{Real2Float}$~\cite{real2float}. Finally, a delicate but important open problem is to apply such optimization techniques in order to handle roundoff errors of programs implementing finite or infinite loops as well as conditional statements.
%
%
\ifCLASSOPTIONcompsoc
\fi
\bibliographystyle{IEEEtran}

\begin{thebibliography}{10}
\providecommand{\url}[1]{#1}
\csname url@samestyle\endcsname
\providecommand{\newblock}{\relax}
\providecommand{\bibinfo}[2]{#2}
\providecommand{\BIBentrySTDinterwordspacing}{\spaceskip=0pt\relax}
\providecommand{\BIBentryALTinterwordstretchfactor}{4}
\providecommand{\BIBentryALTinterwordspacing}{\spaceskip=\fontdimen2\font plus
\BIBentryALTinterwordstretchfactor\fontdimen3\font minus
  \fontdimen4\font\relax}
\providecommand{\BIBforeignlanguage}[2]{{%
\expandafter\ifx\csname l@#1\endcsname\relax
\typeout{** WARNING: IEEEtran.bst: No hyphenation pattern has been}%
\typeout{** loaded for the language `#1'. Using the pattern for}%
\typeout{** the default language instead.}%
\else
\language=\csname l@#1\endcsname
\fi
#2}}
\providecommand{\BIBdecl}{\relax}
\BIBdecl

\bibitem{laurent2009sums}
M.~Laurent, ``Sums of squares, moment matrices and optimization over
  polynomials,'' in \emph{Emerging applications of algebraic geometry}.\hskip
  1em plus 0.5em minus 0.4em\relax Springer, 2009, pp. 157--270.

\bibitem{fptaylor}
A.~Solovyev, C.~Jacobsen, Z.~Rakamari\'c, and G.~Gopalakrishnan, ``{Rigorous
  Estimation of Floating-Point Round-off Errors with Symbolic {Taylor}
  Expansions},'' in \emph{Formal Methods}, 2015.

\bibitem{rosa}
E.~Darulova and V.~Kuncak, ``{Towards a Compiler for Reals},'' EPFL, Tech.
  Rep., 2016.

\bibitem{real2float}
V.~Magron, G.~Constantinides, and A.~Donaldson, ``{Certified Roundoff Error
  Bounds Using Semidefinite Programming},'' \emph{ACM Trans. Math. Softw.},
  vol.~43, no.~4, pp. 1--34, 2017.

\bibitem{Constantinides}
D.~Boland and G.~A. Constantinides, ``Automated precision analysis: A
  polynomial algebraic approach,'' in \emph{FCCM'10}, 2010, pp. 157--164.

\bibitem{schweighofer06}
D.~Grimm, T.~Netzer, and M.~Schweighofer, ``A note on the representation of
  positive polynomials with structured sparsity,'' \emph{Archiv der
  Mathematik}, vol.~89, no.~5, pp. 399--403, 2007.

\bibitem{sbsos}
T.~Weisser, J.~B. Lasserre, and K.-C. Toh, ``Sparse-bsos: a bounded degree sos
  hierarchy for large scale polynomial optimization with sparsity,''
  \emph{Mathematical Programming Computation}, pp. 1--32, 2017.

\bibitem{Lasserre2009Moments}
\BIBentryALTinterwordspacing
J.~B. Lasserre, \emph{{Moments, Positive Polynomials and Their Applications}},
  ser. Imperial College Press optimization series.\hskip 1em plus 0.5em minus
  0.4em\relax Imperial College Press, 2009. [Online]. Available:
  \url{http://books.google.nl/books?id=VY6imTsdIrEC}
\BIBentrySTDinterwordspacing

\bibitem{mourrain2009}
B.~Mourrain and J.~P. Pavone, ``Subdivision methods for solving polynomial
  equations,'' \emph{J. Symb. Comput.}, vol.~44, no.~3, pp. 292--306, 2009.

\bibitem{smithThesis}
A.~P. Smith, ``Enclosure methods for systems of polynomial equations and
  inequalities,'' Ph.D. dissertation, 2012.

\bibitem{dreossiHSCC}
T.~Dreossi and T.~Dang, ``Parameter synthesis for polynomial biological
  models,'' in \emph{HSCC}, 2014.

\bibitem{munoz13}
C.~Mu{\~{n}}oz and A.~Narkawicz, ``Formalization of a representation of
  {B}ernstein polynomials and applications to global optimization,'' \emph{J.
  of Auto. Reason.}, vol.~51, no.~2, pp. 151--196, August 2013.

\bibitem{NGSM12}
A.~Narkawicz, J.~Garloff, A.~Smith, and C.~Mu{\~{n}}oz, ``{Bounding the Range
  of a Rational Function over a Box},'' \emph{Reliable Computing}, vol.~17, pp.
  34--39, 2012.

\bibitem{fluctuat}
D.~Delmas, E.~Goubault, S.~Putot, J.~Souyris, K.~Tekkal, and F.~Védrine,
  ``Towards an industrial use of fluctuat on safety-critical avionics
  software,'' in \emph{FMICS}, 2009.

\bibitem{gappa}
M.~Daumas and G.~Melquiond, ``{Certification of Bounds on Expressions Involving
  Rounded Operators},'' \emph{ACM Trans. Math. Softw.}, vol.~37, no.~1, pp.
  2:1--2:20, Jan. 2010.

\bibitem{hollight}
J.~Harrison, ``{HOL Light: A Tutorial Introduction},'' in \emph{FMCAD}, 1996.

\bibitem{CoqProofAssistant}
``{The Coq Proof Assistant},'' 2016, \url{http://coq.inria.fr/}.

\bibitem{arith24}
A.~Rocca, V.~Magron, and T.~Dang, ``{Certified Roundoff Error Bounds using
  Bernstein Expansions and Sparse Krivine-Stengle Representations},'' in
  \emph{{24th IEEE Symposium on Computer Arithmetic}}.\hskip 1em plus 0.5em
  minus 0.4em\relax IEEE, 2017.

\bibitem{deKlerk10Error}
E.~de~Klerk and M.~Laurent, ``{Error Bounds for Some Semidefinite Programming
  Approaches to Polynomial Minimization on the Hypercube},'' \emph{SIAM J. on
  Optimization}, vol.~20, no.~6, pp. 3104--3120, 2010.

\bibitem{IEEE754}
D.~Zuras, M.~Cowlishaw, A.~Aiken, M.~Applegate, D.~Bailey, S.~Bass,
  D.~Bhandarkar, M.~Bhat, D.~Bindel, S.~Boldo \emph{et~al.}, ``Ieee standard
  for floating-point arithmetic,'' \emph{IEEE Std 754-2008}, pp. 1--70, 2008.

\bibitem{berngarloff}
J.~Garloff, ``Convergent bounds for the range of multivariate polynomials,'' in
  \emph{Interval Mathematics 1985}.\hskip 1em plus 0.5em minus 0.4em\relax
  Springer, 1986, pp. 37--56.

\bibitem{krivineanneaux}
J.-L. Krivine, ``Anneaux pr{\'e}ordonn{\'e}s,'' \emph{Journal d'analyse
  math{\'e}matique}, vol.~12, no.~1, pp. 307--326, 1964.

\bibitem{stengle}
G.~Stengle, ``A nullstellensatz and a positivstellensatz in semialgebraic
  geometry,'' \emph{Mathematische Annalen}, vol. 207, no.~2, pp. 87--97, 1974.

\bibitem{bsos}
J.~B. Lasserre, K.-C. Toh, and S.~Yang, ``A bounded degree sos hierarchy for
  polynomial optimization,'' \emph{EURO J. on Comput. Opt.}, pp. 1--31, 2015.

\bibitem{NN94}
Y.~Nesterov and A.~Nemirovski, \emph{{Interior Point Polynomial Methods in
  Convex Programming: Theory and Applications}}.\hskip 1em plus 0.5em minus
  0.4em\relax Philadelphia: {Society for Industrial and Applied Mathematics},
  1994.

\bibitem{ginac}
C.~Bauer, A.~Frink, and R.~Kreckel, ``Introduction to the ginac framework for
  symbolic computation within the c++ programming language,'' \emph{J. Symb.
  Comput.}, vol.~33, no.~1, pp. 1--12, Jan. 2002.

\bibitem{cplex}
{ILOG, Inc}, ``{ILOG CPLEX}: High-performance software for mathematical
  programming and optimization,'' 2006.

\bibitem{lp_compare}
``Decision tree for optimization software,''
  \url{http://plato.la.asu.edu/bench.html}, accessed: 2016-10-18.

\bibitem{YALMIP}
J.~L{\"o}fberg, ``Yalmip : A toolbox for modeling and optimization in
  {MATLAB},'' in \emph{CACSD}, 2004.

\bibitem{SolovyevH13}
A.~Solovyev and T.~C. Hales, ``Formal verification of nonlinear inequalities
  with taylor interval approximations,'' in \emph{NFM 2013}.

\bibitem{Mantzaflaris2010}
A.~Mantzaflaris and B.~Mourrain, \emph{A Subdivision Approach to Planar
  Semi-algebraic Sets}.\hskip 1em plus 0.5em minus 0.4em\relax Springer Berlin
  Heidelberg, 2010, pp. 104--123.

\end{thebibliography}

\appendices
\section{Program Benchmarks}
\label{appendix}
\begin{itemize} 
	\item \new{rigibody1 : $(x_1,x_2,x_3) \mapsto -x_1x_2-2x_2x_3-x_1-x_3$ defined on $[-15,15]^3$.}
	\item \new{rigibody2 : $(x_1,x_2,x_3) \mapsto 2x_1x_2x_3+6x_3^2-x_2^2x_1x_3-x_2$ defined on $[-15,15]^3$.}
	\item \new{kepler0 : $(x_1,x_2,x_3,x_4,x_5,x_6) \mapsto x_2x_5+x_3x_6-x_2x_3-x_5x_6+x_1(-x_1+x_2+x_3-x_4+x_5+x_6)$ defined on $[4,6.36]^6$.}
	\item \new{kepler1 : $(x_1,x_2,x_3,x_4) \mapsto x_1x_4(-x_1+x_2+x_3-x_4)+x_2(x_1-x_2+x_3+x_4)+x_3(x_1+x_2-x_3+x_4)-x_2x_3x_4-x_1x_3-x_1x_2-x_4$ defined on $[4,6.36]^4$.}
	\item \new{kepler2 : $(x_1,x_2,x_3,x_4,x_5,x_6) \mapsto x_1x_4(-x_1+x_2+x_3-x_4+x_5+x_6)+x_2x_5(x_1-x_2+x_3+x_4-x_5+x_6)+x_3x_6(x_1+x_2-x_3+x_4+x_5-x_6)-x_2x_3x_4-x_1x_3x_5-x_1x_2x_6-x_4x_5x_6$ defined on $[4,6.36]^6$.}
	\item \new{sineTaylor : $x \mapsto x-\frac{x^3}{6.0}+\frac{x^5}{120.0}-\frac{x^7}{5040.0}$ defined on $[-1.57079632679,1.57079632679]$.}
	\item \new{sineOrder3 : $x \mapsto 0.954929658551372 x-0.12900613773279798 x^3$ defined on $[-2,2]$.}
	\item \new{sqroot : $x \mapsto 1.0+0.5x-0.125x^2+0.0625x^3-0.0390625x^4$ defined on $[0,1]$.}
	\item \new{himmilbeau : $(x_1,x_2) \mapsto (x_1^2+x_2-11)^2+(x_1+x_2^2-7)^2$ defined on $[-5,5]^2$.}
	\item \new{schwefel : $(x_1,x_2,x_3) \mapsto (x_1-x_2)^2+(x_2-1)^2+(x_1-x_3^2)^2+(x_3-1)^2$ defined on $[-10,10]^3$.}
	\item \new{magnetism : $(x_1,x_2,x_3,x_4,x_5,x_6,x_7) \mapsto x_1^2+2 x_2^2+2 x_3^2+2 x_4^2+2 x_5^2+2 x_6^2+2 x_7^2-x_1$ defined on $[-1,1]^7$.}
	\item \new{caprasse : $(x_1,x_2,x_3,x_4) \mapsto x_1 x_3^3+4 x_2 x_3^2 x_4+4 x_1 x_3 x_4^2+2 x_2 x_4^3+4 x_1 x_3+4 x_3^2-10 x_2 x_4-10 x_4^2+2$ defined on $[-0.5,0.5]^4$.}
	\item \new{doppler1 : $(x_1,x_2,x_3) \mapsto -t_1x_2/((t_1+x_1)(t_1+x_1))$ defined on $[-100,100]\times[20,20000]\times[-30,50]$, with $t_1 = 331.+0.6x_3$.}
	\item \new{doppler2 : $(x_1,x_2,x_3) \mapsto -t_1x_2/((t_1+x_1)(t_1+x_1))$ defined on $[-125,125]\times[15,25000]\times[-40,60]$, with $t_1 = 331.+0.6x_3$.}
	\item \new{doppler3 : $(x_1,x_2,x_3) \mapsto -t_1x_2/((t_1+x_1)(t_1+x_1))$ defined on $[-300,120]\times[320,20300]\times[-50,30]$, with $t_1 = 331.+0.6x_3$.}
	\item \new{verhulst : $x \mapsto 4x/(1+\frac{100}{111}\,x)$ defined on $[0.1,0.3]$.}
	\item \new{carbonGas : $x \mapsto (p+a(n/x)^2)(x-n\,b)-1.3806503\text{e-23}\,n\,t$ defined on $[0.1,0.5]$, with $p=3.5\text{e-7};a=0.401;b=42.7\text{e-7};t=300;n=1000$.}
	\item \new{predPrey : $x \mapsto 4xx/(1+(\frac{100}{111}\,x)^2)$ defined on $[0.1,0.3]$.}
	\item \new{turbine1 : $(x_1,x_2,x_3) \mapsto (3+2/(x_3x_3)-0.125(3-2x_1)(x_2x_2x_3x_3)/(1-x_1) -4.5)$ defined on $[-4.5,-0.3]\times[0.4,0.9]\times[3.8,7.8]$.}
	\item \new{turbine2 : $(x_1,x_2,x_3) \mapsto 6x_1 - 0.5x_1(x_2x_2x_3x_3)/(1-x_1)-2.5$ defined on $[-4.5,-0.3]\times[0.4,0.9]\times[3.8,7.8]$.}
	\item \new{turbine3 : $(x_1,x_2,x_3) \mapsto 3-2/(x_3x_3)-0.125(1+2x_1)(x_2x_2x_3x_3)/(1-x_1) - 0.5$ defined on $[-4.5,-0.3]\times[0.4,0.9]\times[3.8,7.8]$.}
	\item \new{jet : $(x_1,x_2) \mapsto x_1+((2x_1((3x_1x_1+2x_2-x_1)/(x_1x_1+1))((3x_1x_1+2x_2-x_1)/(x_1x_1+1)-3)+
		x_1x_1(4((3x_1x_1+2x_2-x_1)/(x_1x_1+1))-6))(x_1x_1+1)+3x_1x_1((3x_1x_1+2x_2-x_1)/(x_1x_1+1))+x_1x_1x_1+x_1+3((3x_1x_1+2x_2-x_1)/(x_1x_1+1)))$ defined on $[-5,5]\times[-20,20]$.}
	\item \new{floudas2-6 : $(x_1,x_2,x_3,x_4,x_5,x_6,x_7,x_8,x_9,x_10) \mapsto 48x_1+42x_2+48x_3+45x_4+44x_5+41x_6+57x_7+42x_8+45x_9+46x_{10}-50(x_1x_1+x_2x_2+x_3x_3+x_4x_4+x_5x_5+x_6x_6+x_7x_7+x_8x_8 + x_9x_9+x_{10}x_{10}$ defined on $[0,1]^{10}$ and the constraints set:\\
		$\{0\leq (-4+2x_1+6x_2+1x_3+0x_4+3x_5+3x_6+2x_7+6x_8+2x_9+2x_{10});$\\
        $0\leq 22-(6x_1-5x_2+8x_3-3x_4+0x_5+1x_6+3x_7+8x_8+9x_9-3x_{10});$\\
        $0\leq -6-(5x_1+6x_2+5x_3+3x_4+8x_5-8x_6+9x_7+2x8_+0x_9-9x_{10});$\\
        $0\leq -23-(9x_1+5x_2+0x_3-9x_4+1x_5-8x_6+3x_7-9x_8-9x_9-3x_{10});$\\
        $0\leq -12-(-8x_1+7x_2-4x_3-5x_4-9x_5+1x_6-7x_7-1x_8+3x_9-2x_{10})\}$ }
    \item \new{floudas3-3 : $(x_1,x_2,x_3,x_4,x_5,x_6) \mapsto -25(x_1-2)^2-(x_2-2)^2-(x_3-1)^2-(x_4-4)^2-(x_5-1)^2-(x_6-4)^2$ defined on $[0,6]^2\times[1,5]\times[0,6]\times[1,5]\times[0,10]$ and the constraints set:\\
		$\{0\leq ((x_3-3)^2+x_4-4);$\\
        $0\leq ((x_5-3)^2+x_6-4);$\\
        $0\leq (2-x_1+3x_2);$\\
        $0\leq (2+x1-x_2);$\\
        $0\leq (6-x_1-x_2);$\\
        $0\leq (x_1+x_2-2)\}$ }
    \item \new{floudas3-4 : $(x_1,x_2,x_3) \mapsto -2x_1+x_2-x_3$ defined on $[0,2]^2\times[0,3]$ and the constraints set:\\
		$\{0\leq (4-x_1-x_2-x_3);$\\
        $0\leq (6-3x_2-x_3);$\\
        $0\leq (-0.75+2x_1-2x_3+4x_1x_1-4x_1x_2+4x_1x_3+2x_2x_2-2x_2x_3+2x_3x_3) \}$ }
    \item \new{floudas4-6 : $(x_1,x_2) \mapsto -x_1-x_2$ defined on $[0,3]\times[0,4]$ and the constraints set:\\
		$\{0\leq (2x_1^4-8x_1^3+8x_1x1-x_2);$\\
		$0\leq (4x_1^4-32x_1^3+88x_1x_1-96x_1+36-x_2)\}$ }
    \item \new{floudas4-7 : $(x_1,x_2) \mapsto -12x_1-7x_2+x_2x_2$ defined on $[0,2]\times[0,3]$ and the constraints set: $\{0\leq (-2x_1^4+2-x_2)\}$ }
\end{itemize} 

\end{document}